\documentclass[11pt]{article}
\usepackage[a4paper,margin=1in]{geometry}
\usepackage{amsmath,amssymb,amsthm,mathtools}
\usepackage{bm,booktabs,array,enumitem,microtype}
\usepackage[dvipsnames]{xcolor}
\usepackage{tikz}
\usepackage{centernot}
\usepackage{mathrsfs}
\usepackage{graphicx}
\usepackage{float}
\usepackage{pgfplots}
\usepackage{authblk}
\pgfplotsset{compat=1.18}
\usetikzlibrary{arrows.meta,calc,positioning,shapes,decorations.pathreplacing,patterns,backgrounds}
\usepackage[colorlinks=true,linkcolor=MidnightBlue,citecolor=BrickRed,urlcolor=RoyalBlue]{hyperref}

\newtheorem{theorem}{Theorem}[section]
\newtheorem{proposition}[theorem]{Proposition}
\newtheorem{lemma}[theorem]{Lemma}

\theoremstyle{definition}
\newtheorem{definition}[theorem]{Definition}
\newtheorem{assumption}[theorem]{Assumption}
\newtheorem{remark}[theorem]{Remark}

\newcommand{\E}{\mathbb E}
\newcommand{\Pp}{\mathbb P}
\newcommand{\R}{\mathbb R}
\newcommand{\N}{\mathbb N}
\newcommand{\Z}{\mathbb Z}
\newcommand{\D}{\mathcal D}
\newcommand{\Lop}{\mathcal L}
\newcommand{\Gop}{\mathcal G}
\newcommand{\cF}{\mathcal F}
\newcommand{\cM}{\mathcal M}
\newcommand{\cP}{\mathcal P}
\DeclareMathOperator{\Law}{Law}
\newcommand{\ind}{\mathbf 1}
\newcommand{\diff}{\mathrm{d}}

\title{Jump limits of continuous multiscale diffusions:\\ Homogenization under heavy-tailed activity bursts}

\author{Qi Zhang}
\affil{Beijing institute of mathematical sciences and applications, Beijing, 101408 China, Email: \texttt{qzhang@bimsa.cn}}

\date{September 2026}

\begin{document}
\maketitle

\begin{abstract}
We show that purely continuous microscopic diffusions can homogenize to a Markov process with genuine jumps. Consider a multiscale diffusion with time-dependent generator $\mathcal{L}_{0}+r_{t}^{\epsilon}\mathcal{L}_{1}$ driven by a regenerative random environment. When the integrated activity over a regeneration cycle has a regularly varying tail of index $\beta\in(0,1)$, the cumulative operational clock converges to a $\beta$-stable subordinator, compressing finite activity-driven diffusion into vanishingly short physical time intervals. 
To overcome the noncommutativity of $\mathcal{L}_0$ and $\mathcal{L}_1$, we develop an extended-time parametrization paired with Khasminskii-type time discretization, establishing finite-dimensional convergence and strong semigroup convergence on $C_{0}(\mathbb{R}^{d})$ toward a jump process with nonlocal generator $\overline{\mathcal{L}}=\mathcal{L}_{0}-\kappa(-\mathcal{L}_{1})^{\beta}$. Under uniform ellipticity of $\mathcal{L}_{1}$, the fractional power $(-\mathcal{L}_{1})^{\beta}$ is identified as a $2\beta$-stable-like nonlocal operator. We also clarify physical-time path-topology obstructions and apply the framework to Poisson resetting pulses, Ornstein--Uhlenbeck diffusions, and multiscale stochastic volatility.
\end{abstract}

\section{Introduction}
\label{sec:intro}

The asymptotic analysis of multiscale stochastic differential equations (SDEs) plays a central role in modern applied probability and mathematical physics. Classical averaging and homogenization theories describe a rapidly fluctuating environment through its macroscopic effective action on a slow system. When the fast environmental observable has sufficient mixing and integrable cycle rewards, its accumulated contribution converges to a deterministic linear clock, and the Fredholm alternative or Poisson equation then yields an effective continuous diffusion \cite{PardouxVeretennikov2001,PardouxVeretennikov2003}.

Heavy-tailed activity leads to a fundamentally different question: \emph{can a continuous microscopic diffusion converge to an effective process with jumps because a finite amount of motion is compressed into a vanishing interval of physical time?} In this paper, we demonstrate that this anomalous mechanism occurs when the operational activity clock converges to a stable subordinator, giving rise to an effective evolution characterized by generalized stable-like nonlocal operators.

To formulate this problem precisely, let $(\Omega,\cF,\Pp)$ be the probability space of the random environment equipped with a measure-preserving ergodic flow $(\theta_t)_{t\in\R}$, and let $Q:\Omega\to[0,\infty)$ be a nonnegative observable. We consider the multiscale diffusion
\begin{equation}\label{eq:micro_sde}
 \diff X_t^\varepsilon = b_0(X_t^\varepsilon)\,\diff t + \sigma_0(X_t^\varepsilon)\,\diff W_t^0  + r_t^\varepsilon\,b_1(X_t^\varepsilon)\,\diff t + \sqrt{r_t^\varepsilon}\,\sigma_1(X_t^\varepsilon)\,\diff W_t^1,
\end{equation}
where $r_t^\varepsilon(\omega):=Q(\theta_{t/\varepsilon}\omega)/(\varepsilon\,a_{\varepsilon^{-1}})$ is the rescaled random activity rate, and $W^0,W^1$ are independent Brownian motions on a separate probability space $(\Omega_W,\mathcal F_W,(\mathcal F_t^W)_{t\ge0},\Pp_W)$. The system couples a continuous background diffusion generated by $\Lop_0$ with an intermittent activity diffusion generated by $\Lop_1$, giving the conditional time-dependent generator $\Lop_t^\varepsilon(\omega)=\Lop_0+r_t^\varepsilon\Lop_1$. The cumulative environmental activity serves as the operational clock:
\begin{equation}\label{eq:micro_gen}
 A_t^\varepsilon
 =\int_0^t r_s^\varepsilon\,\diff s
 =\frac{1}{a_{\varepsilon^{-1}}}\int_0^{t/\varepsilon}Q(\theta_s\omega)\,\diff s,
 \qquad \mu^\varepsilon=\diff A^\varepsilon.
\end{equation}
Physical time $t$ measures chronological duration, while $A_t^\varepsilon$ records the total activity-driven motion. When the activity rate exhibits intense, concentrated bursts, an episode of vanishing physical duration can produce a macroscopic advance in operational time.

In standard homogenization, the ergodic theorem implies $A_t^\varepsilon/t\to \E[Q]$ whenever $\E[Q]<\infty$, yielding another diffusion in the limit. The situation changes radically when the integrated activity over a regeneration cycle, $R_k=\int_{\tau_{k-1}}^{\tau_k}Q(\theta_s\omega)\,\diff s$, has a heavy tail of regular variation with index $\beta\in(0,1)$. 
Since $\E[R_k]=\infty$, ergodic averaging breaks down, and the normalized clock $A^\varepsilon$ converges instead to a strictly $\beta$-stable subordinator $S^\beta$. 
Such stable limits for heavy-tailed additive functionals of Markov processes and stationary sequences have been studied in
\cite{AaronsonZweimuller2014,JaraKomorowskiOlla2009,Bethencourt2023,TyranKaminska2010,
BasrakKrizmanicSegers2012}.
In the pure-activity Brownian setting ($\Lop_0=0$, $\sigma_1=I$), conditioning on the clock shows that
\[
 \E e^{i\xi\cdot X_t^\varepsilon}
 =\E e^{-\frac12|\xi|^2 A_t^\varepsilon}
 \longrightarrow e^{-t\kappa(|\xi|^2/2)^\beta},
\]
identifying the limit $B_{S_t^\beta}$ as a rotationally symmetric $2\beta$-stable L\'evy process with generator $-\kappa(-\tfrac{1}{2}\Delta)^\beta$ \cite{Bertoin1996,Sato1999,SchillingSongVondracek2012,CaffarelliSilvestre2007}. At the microscopic level, an activity burst of operational size $r$ compresses the continuous heat semigroup $P_r^{(1)}$ into an infinitesimal slice of physical time. 
Integrating the Gaussian kernel against the stable jump density produces the power-law profile  
\[ \int_0^\infty(2\pi r)^{-d/2}e^{-|z|^2/(2r)}r^{-1-\beta}\,\diff r \propto |z|^{-d-2\beta}, \]
showing how purely continuous local diffusion transitions combine into macroscopic jump dynamics.

While the pure-activity setting can be analyzed via classical subordination, the general multiscale system~\eqref{eq:micro_sde} presents substantial analytical challenges because $\Lop_0$ and $\Lop_1$ generally do not commute. The slow system is not frozen between bursts; it continues to evolve under the persistent background diffusion $\Lop_0$ in physical time $t$ while burst transitions occur under $\Lop_1$. When macroscopic bursts of operational sizes $J_j$ arrive at physical times $0<T_1<\cdots<T_m<t$, the conditional transition operator must interleave both semigroups combined in time order:
\[
 P_{T_1}^{(0)}P_{J_1}^{(1)}
 P_{T_2-T_1}^{(0)}P_{J_2}^{(1)}\cdots
 P_{T_m-T_{m-1}}^{(0)}P_{J_m}^{(1)}P_{t-T_m}^{(0)}.
\]
Because the two generators do not commute, subordinating $P^{(1)}$ alone cannot construct the full macroscopic evolution.

To resolve this time-ordering obstruction, we augment the spatial diffusion with the accumulated operational activity. In the extended state space $\R^d\times[0,\infty)$, the coupled process $(X_t^\varepsilon, A_t^\varepsilon)$ advances along continuous, $1$-Lipschitz paths, allowing the limiting dynamics to be constructed as the solution of an SDE driven by a Poisson random measure with bounded Lipschitz coefficients. Adapting Khasminskii's time-discretization method \cite{KhasminskiiAveraging1968} to extended time, we freeze the coefficients on each mesh interval; uniform moment bounds and Euler approximation estimates then yield tightness and convergence of the extended-time processes. This extended-time formulation constructs the limiting Markov semigroup and verifies its Feller property before using the Phillips functional calculus to identify its generator. Moreover, controlling the accumulation of small bursts as $r\to 0$ requires the Skorokhod $M_1$ topology on $A^\varepsilon$, which is robust under composition with continuous functions and avoids the topological obstructions of the $J_1$ metric when continuous paths converge to discontinuous limits.

This approach yields a complete homogenization theory across three levels. First, Theorem~\ref{thm:clock_limit} establishes the environmental clock limit: under regenerative conditions with i.i.d.\ complete cycles, regularly varying cycle rewards of index $\beta\in(0,1)$, and finite mean cycle duration $\mu_\tau<\infty$, the activity measure converges vaguely, $\mu^\varepsilon\Rightarrow\diff S^\beta$, and the operational clock converges in $(D([0,\infty)),M_1)$ to a strictly $\beta$-stable subordinator $S^\beta$. Formulating this criterion through integrated cycle rewards rather than the pointwise marginal of $Q$ accommodates reward-duration dependence and delayed stationary initialization. Second, Theorem~\ref{thm:full_homogenization} transfers this clock convergence to the multiscale diffusion: for bounded globally Lipschitz coefficients, we prove annealed finite-dimensional convergence and strong semigroup convergence on $C_0(\R^d)$ toward a conservative Feller process. The limiting generator on $C_c^\infty(\R^d)$ takes the local-to-nonlocal form
\begin{equation} \label{eq:central_chain}
 \overline{\Lop}f=\Lop_0f+  \int_0^\infty(P_r^{(1)}f-f)\,\nu_\beta(\diff r) =\Lop_0f-\kappa(-\Lop_1)^\beta f, \qquad \kappa=\Gamma(1-\beta)/\mu_\tau.
\end{equation}
Third, under uniform ellipticity of $\Lop_1$, Theorem~\ref{cor:fractional_homog} identifies $(-\Lop_1)^\beta$ as a generalized $2\beta$-stable-like nonlocal operator and derives two-sided local jump kernel bounds $c|x-y|^{-d-2\beta}\le j_{2\beta}(x,y)\le C|x-y|^{-d-2\beta}$, showing how the spatial inhomogeneity of the microscopic diffusion dictates the macroscopic jump geometry.

These findings differ fundamentally from classical homogenization theories \cite{Khasminskii1968,PapanicolaouStroockVaradhan1977,BensoussanLionsPapanicolaou1978,JikovKozlovOleinik1994,PavliotisStuart2008} and invariance principles in random media \cite{KipnisVaradhan1986,Olla1994}, where fast fluctuations average out through Poisson equations \cite{PardouxVeretennikov2001,PardouxVeretennikov2003} and produce local second-order diffusions. In our heavy-tailed regime, the infinite-mean cycle rewards break the classical Fredholm alternative, generating an effective evolution with jumps. Although multiscale systems driven by microscopic jump noise have been investigated by Franke \cite{Franke2006} and Sandri\'c \cite{Sandric2016}, their limits merely preserve and renormalize pre-existing nonlocality. Here, the microscopic system is driven solely by Brownian motion; the nonlocal limit is an emergent phenomenon created entirely by the temporal compression of environmental bursts.

The mechanism also reveals an intriguing duality with continuous-time random walks and subdiffusive transport \cite{BouchaudGeorges1990,MetzlerKlafter2000,MeerschaertSikorskii2012}. While heavy-tailed waiting times produce an \emph{inverse} subordinator that traps and retards the system, heavy-tailed activity rewards generate a \emph{direct} subordinator that accelerates operational time, transforming continuous diffusions into jump processes. Stochastic clocks have also been widely utilized in financial econometrics \cite{ContTankov2004} and operator theory \cite{Phillips1952,Bertoin1996,Sato1999}, but those frameworks typically postulate the subordinated clock as an \emph{a priori} model. In contrast, we derive the subordinated semigroup dynamically from an underlying multiscale SDE in the presence of an active background operator $\Lop_0$.

Furthermore, the limiting generator $\overline{\Lop}=\Lop_0-\kappa(-\Lop_1)^\beta$ connects our probabilistic homogenization with the theory of variable-coefficient stable-like operators \cite{Bass1988,Kolokoltsov2000,Jacob2001,ChenKumagai2003,SongVondracek2003,BottcherSchillingWang2013}. While those works study analytic and pseudo-differential properties starting from a given jump kernel, our approach derives both the generator and its kernel bounds from the heat kernel of the microscopic activity diffusion and regenerative limit theorems \cite{AaronsonZweimuller2014,BasrakKrizmanicSegers2012,MelbourneZweimuller2015,Bethencourt2023}, thereby providing a physical and dynamical origin for this class of nonlocal operators.

The remainder of the paper is organized as follows. Section~\ref{sec:preliminaries} reviews the stationary environmental setting alongside tools from L\'evy--Poisson calculus and Feller semigroups. Section~\ref{sec:environmental_criteria} establishes the regenerative clock convergence theorem and verifies it on concrete models, including Poisson pulses and Ornstein--Uhlenbeck diffusions. Section~\ref{sec:model} presents the precise assumptions, the homogenization theorem, and applications to stochastic volatility. The proofs of homogenization and the jump kernel estimates are detailed in Section~\ref{sec:proofs}. Appendix~\ref{app:measure_topology} collects background on the measure-clock topologies.

\paragraph{Notation.}
\label{par:notation}
For a locally compact space $E$, $C_0(E)$ denotes the continuous functions vanishing at infinity, and $C_c^k(\R^m)$ denotes the $k$-times continuously differentiable functions with compact support. We write $\D=C_c^\infty(\R^d)$ for the test-function class and $D(\Lop)$ for the domain of a generator $\Lop$. For a bounded function $f:E\to\R$, $\|f\|_\infty:=\sup_{x\in E}|f(x)|$; for a bounded linear operator $T:C_0(E)\to C_0(E)$, $\|T\|_{\mathrm{op}}:=\sup_{\|f\|_\infty\le1}\|Tf\|_\infty$. Strong operator convergence means $\|T_n f-Tf\|_\infty\to0$ for every fixed $f$. For a path $z$ on a compact interval $J$, $\|z\|_{\infty,J}:=\sup_{u\in J}|z(u)|$. Unless otherwise specified, $|\cdot|$ denotes the Euclidean norm; for matrices we use the Frobenius norm and $A:B:=\operatorname{tr}(A^\top B)$. For random variables, $\|Y\|_{L^p}=(\E|Y|^p)^{1/p}$.

\section{Preliminaries}
\label{sec:preliminaries}

We first describe the random environment and its operational clock. The topology for activity measures and its relation to clock convergence are recalled in Appendix~\ref{app:measure_topology}. We then collect the L\'evy--Poisson calculus and Feller semigroup results used to construct and identify the limiting dynamics.

\subsection{Random environments and operational time}
\label{subsec:random_environment}

We begin by specifying the stationary framework of the random medium. Let $(\Omega, \cF, \Pp)$ be a complete probability space equipped with a one-parameter family of measure-preserving transformations $(\theta_t)_{t\in\R}$ forming a stationary ergodic metric dynamical system. Thus,
\begin{enumerate}[label=(\roman*)]
\item $\theta_0 = \mathrm{id}_\Omega$ and $\theta_{t+s} = \theta_t \circ \theta_s$ for all $s, t \in \R$;
\item the mapping $(t,\omega) \mapsto \theta_t\omega$ is jointly $(\mathcal B(\R)\otimes\cF, \cF)$-measurable;
\item $\Pp(\theta_t^{-1}A) = \Pp(A)$ for all $A \in \cF$ and $t \in \R$ (stationarity);
\item if $\theta_t^{-1}A = A$ for all $t \in \R$, then $\Pp(A) \in \{0,1\}$ (ergodicity).
\end{enumerate}

By stationary ergodic metric dynamical system $(\theta_t)_{t\in\R}$, the nonnegative random environmental observable $Q:\Omega\to[0,\infty)$ determines the activity $Q(\theta_t\omega)$. The following assumptions ensure that the accumulated activity is finite on bounded intervals and specify the regenerative structure used to obtain its stable limit.

\begin{assumption}[Locally finite activity]
\label{ass:local_activity}
The nonnegative observable $Q$ is measurable and satisfies
$\int_0^H Q(\theta_s\omega)\,\diff s<\infty$ almost surely for every finite $H$.
\end{assumption}

\begin{assumption}[Regenerative conditions]
\label{ass:level_C}
The stationary ergodic environment satisfies:
\begin{enumerate}[label=(\alph*)]
\item There exist random regeneration times $0\le\tau_0<\tau_1<\cdots$, with $\tau_0<\infty$ almost surely, so that $\mu_\tau=\E[\tau_k-\tau_{k-1}]\in(0,\infty), k\ge1.$
Moreover, the complete cycle segments 
\[
 \mathcal C_k:=\left(\tau_k-\tau_{k-1},\bigl(Q(\theta_{\tau_{k-1}+s}\omega)\bigr)_{0\le s<(\tau_k-\tau_{k-1})}\right),
 \qquad k\ge1,
\]
are i.i.d.\ and independent of the initial segment $(\tau_0,(Q(\theta_s\omega))_{0\le s<\tau_0})$. 
\item The complete cycle rewards
$R_k=\int_{\tau_{k-1}}^{\tau_k}Q(\theta_s\omega)\,\diff s$, $k\ge1$ satisfy the heavy tail
\begin{equation}\label{eq:regular_variation}
 \Pp(R_1>x)=x^{-\beta}L_R(x),\qquad 0<\beta<1,
\end{equation}
where $L_R$ is a positive measurable function on $(0,\infty)$, so that $L_R(cx)/L_R(x)\to1$ as $x\to\infty$ for every $c>0$. The normalization scale $a_T\to\infty$ so that $T\Pp(R_1>a_T)\to1$.
\end{enumerate}
\end{assumption}

For the complete-cycle reward $R_1$, regular variation and the normalization in Assumption~\ref{ass:level_C}(b) give $T\Pp(R_1>a_T r)\longrightarrow r^{-\beta}, r>0$.
The tail of this integrated reward depends on both the magnitude and persistence of activity within a cycle. The examples in Section~\ref{sec:verification} verify the criterion by computing complete-cycle rewards. Their negligible macroscopic duration follows from the integrability of a complete-cycle length.

The scale $a_T$ in Assumption~\ref{ass:level_C}(b) normalizes the integrated cycle rewards: among order $T$ complete cycles, order one rewards exceed $a_T$. Thus $a_{\varepsilon^{-1}}$ is this reward scale evaluated at the fast-time scale $T=\varepsilon^{-1}$. With this normalization, we define
\begin{equation}
 \mu^\varepsilon(\diff t)
 =r_t^\varepsilon\,\diff t,
 \qquad
 A_t^\varepsilon=\mu^\varepsilon([0,t])
 =\frac1{a_T}\int_0^{Tt}Q(\theta_s\omega)\,\diff s.
 \label{eq:cumulative_clock}
\end{equation}
These expressions are defined by local integrability under Assumption~\ref{ass:local_activity}. Physical time measures the duration of an episode, whereas $A^\varepsilon$ records how much activity-driven motion it produces. A burst can have vanishing physical duration and a nonvanishing increment of operational time. Its limiting effect will be a full activity-semigroup transition, combined in time order with the background motion. The accumulated clock is used directly; it is not an inverse clock describing trapping or waiting times.

\subsection{L\'evy processes, Poisson calculus, and semigroups}
\label{subsec:levy_semigroups}

We now collect the probabilistic and nonlocal operator facts needed to describe the limiting clock and the transitions generated by its bursts.
First, regarding the underlying jump noise and stochastic integration, a process $Z=(Z_t)_{t\ge0}$ in $\R^m$ is a L\'evy process if $Z_0=0$ almost surely, its increments are stationary and independent, and it is stochastically continuous. We use its c\`adl\`ag version. With the truncation $z\ind_{\{|z|\le1\}}$, its characteristic triplet is $(b_L,A_L,\nu_L)$, where $b_L\in\R^m$, $A_L$ is symmetric nonnegative definite, and
\[
 \nu_L(\{0\})=0,\qquad
 \int_{\R^m}(1\wedge|z|^2)\,\nu_L(\diff z)<\infty.
\]
The associated L\'evy--Khintchine convention is
\begin{equation}\label{eq:levy_khintchine_prelim}
  \E e^{i\xi\cdot Z_t}  = \exp\Big\{-t \big( -ib_L\cdot\xi+\tfrac12\xi^\top A_L\xi +\int_{\R^m}\bigl(1-e^{i\xi\cdot z} +i\xi\cdot z\ind_{\{|z|\le1\}}\bigr)\,\nu_L(\diff z) \big)\Big\}.
\end{equation}
A Poisson random measure $N$ on $[0,\infty)\times(\R^m\setminus\{0\})$ with intensity $\diff t\,\nu_L(\diff z)$ assigns independent Poisson counts to disjoint measurable sets of finite intensity. Its compensated version is denoted by $\widetilde N=N-\diff t\,\nu_L$. If $H$ is predictable and square integrable against the intensity, then
\begin{equation}
 \E\left|\int_0^t\int H(s,z)\,\widetilde N(\diff s,\diff z)\right|^2
 =\E\int_0^t\int|H(s,z)|^2\,\nu_L(\diff z)\,\diff s.
 \label{eq:poisson_isometry}
\end{equation}
For nonnegative predictable $H$, the expectation of its integral against $N$ equals its integral against the intensity. This compensation identity also holds when both sides are infinite.

Writing $A_L=\Sigma_L\Sigma_L^\top$, the L\'evy--It\^o representation is
\begin{equation}
 Z_t=b_Lt+\Sigma_LW_t
 +\int_0^t\int_{|z|\le1}z\,\widetilde N(\diff s,\diff z)
 +\int_0^t\int_{|z|>1}z\,N(\diff s,\diff z),
 \label{eq:levy_ito_prelim}
\end{equation}
where $W$ and $N$ are independent. The small-jump integral is the limit in $L^2$ of integrals with a positive lower cutoff; the large-jump integral is a finite sum on bounded time intervals. No first moment of the large jumps is assumed. We use these standard facts with the conventions of \cite{Applebaum2009,Sato1999}.

For $Z_t^x=x+Z_t$ and $f\in C_c^2(\R^m)$, we introduce the generator of the L\'evy process
\begin{equation}\label{eq:levy_generator_prelim}
 \begin{split}
 \Gop_Lf(x)={}&b_L\cdot\nabla f(x)+\tfrac12 A_L:D^2f(x)\\
 &+\int_{\R^m}\bigl[f(x+z)-f(x)
       -\ind_{\{|z|\le1\}}z\cdot\nabla f(x)\bigr]\,\nu_L(\diff z).
 \end{split}
\end{equation}
Taylor's formula bounds the small-jump integrand by a constant times $|z|^2$, and boundedness of $f$ controls the large jumps. 
Then the It\^o's formula reads
\begin{equation}
 \begin{split}
 f(Z_t^x)-f(x)={}&\int_0^t\Gop_L f(Z_{s-}^x)\,\diff s
 +\int_0^t\nabla f(Z_{s-}^x)\Sigma_L\,\diff W_s\\
 &\qquad +\int_0^t\int_{\R^m}
       [f(Z_{s-}^x+z)-f(Z_{s-}^x)]\,\widetilde N(\diff s,\diff z).
 \end{split}
 \label{eq:ito_levy_prelim}
\end{equation}
Both stochastic integrals are square-integrable martingales on bounded time intervals: their integrands are bounded by a constant times $1\wedge|z|$ for the jump part and by a constant for the Brownian part. The formula follows first with jumps truncated away from zero from the continuous It\^o formula and the finite jump increments. Letting the cutoff decrease to zero, the Poisson isometry controls the compensated martingales and Taylor's quadratic remainder controls their compensators. Taking expectations yields Dynkin's identity.

To analyze the resulting Markov evolutions on continuous functions, we turn to Feller semigroups and their generators. We equip $C_0(\R^m)$ with $\|\cdot\|_\infty$. Associated with the L\'evy process $Z=(Z_t)_{t\ge0}$ in $\R^m$, the transition operators
\begin{equation}
 P_t f(x) = \E[f(x + Z_t)], \qquad t\ge0, \quad x\in\R^m,
 \label{eq:levy_semigroup_def}
\end{equation}
define a conservative Feller semigroup on $C_0(\R^m)$. Each $P_t$ is a positive linear contraction, and stationary independent increments give $P_0=I$ and $P_{t+s}=P_tP_s$. Conservativity means $P_t1=1$ for the transition kernel's natural extension to bounded measurable functions; the constant function $1$ is not in $C_0(\R^m)$. Continuity and vanishing at infinity follow by dominated convergence. Uniform continuity of $f\in C_0(\R^m)$ and stochastic continuity of $Z$ give
\begin{equation}
 \|P_t f - f\|_{\infty} \le \E\!\left[\sup_{x\in\R^m}|f(x + Z_t) - f(x)|\right] \longrightarrow 0 \qquad \text{as } t\downarrow0.
 \label{eq:levy_strong_continuity}
\end{equation}

The infinitesimal generator $\Lop$ of a Feller semigroup $(P_t)_{t\ge0}$ on $C_0(\R^m)$ is defined by the limit in $\|\cdot\|_\infty$
\[
 \Lop f = \lim_{t\downarrow0}\frac{P_t f - f}{t}, \qquad D(\Lop) = \left\{f\in C_0(\R^m) : \lim_{t\downarrow0}\frac{P_t f - f}{t} \text{ exists in } C_0(\R^m)\right\}.
\]
For the L\'evy semigroup \eqref{eq:levy_semigroup_def}, taking expectations in \eqref{eq:ito_levy_prelim} gives $P_tf-f=\int_0^tP_s\Gop_Lf\,\diff s$. For $f\in C_c^2(\R^m)$, Taylor's bound on the small jumps and dominated convergence give $\Gop_Lf\in C_0(\R^m)$. Strong continuity then yields $t^{-1}(P_tf-f)\to\Gop_Lf$ in $\|\cdot\|_\infty$. Thus $C_c^2(\R^m)\subset D(\Lop)$, with generator given by \eqref{eq:levy_generator_prelim}.

A fundamental property of any strongly continuous contraction semigroup on $C_0(\R^m)$ (see e.g. \cite[Lemma~II.1.3(iv), p.~50]{EngelNagel2000}) is the following quantitative increment bound, which will be applied repeatedly to control the transitions of the activity diffusion across compressed burst episodes.

\begin{lemma}\label{ass:semigroup_lipschitz}
Let $(P_t)_{t\ge0}$ be a strongly continuous contraction semigroup on $C_0(\R^m)$ with generator $\Lop$. For every $f\in D(\Lop)$ and $s\ge0$, $\|P_sf-f\|_{\infty}\le s\|\Lop f\|_{\infty}.$
\end{lemma}

Under Assumption~\ref{ass:slow_coefficients}, the background and activity diffusions generate conservative Feller semigroups $P^{(0)}$ and $P^{(1)}$ on $C_0(\R^d)$, as verified in Section~\ref{sec:model}. Even with spatially varying activity coefficients, the estimate $\|P_s^{(1)}f-f\|_{\infty}\le s\|\Lop_1f\|_{\infty}$ holds for $f\in D(\Lop_1)$. It controls changes in expected observables and ensures integrability in the Phillips formula and convergence of truncated generators. Control of the slow paths requires the stability argument in Section~\ref{sec:proofs}.

To model nondecreasing operational clocks, we now focus on subordinators and their Laplace exponents. A subordinator is a L\'evy process with nondecreasing paths, started at zero. Without killing, its Laplace transform has the form $\E[e^{-\lambda S_t}]=e^{-t\Phi(\lambda)}$, where $\Phi(0+)=0$. Its Laplace exponent $\Phi$ is a Bernstein function, defined as follows.

\begin{definition}[Bernstein function]
\label{def:bernstein}
A smooth function $\Phi: (0,\infty) \to [0,\infty)$ is called a \emph{Bernstein function} if $(-1)^{n-1}\Phi^{(n)}(\lambda) \ge 0$ for all $n \in \N$ and $\lambda > 0$. By the L\'evy--Khintchine representation, every Bernstein function with $\Phi(0+) = 0$ can be written as
\begin{equation}
 \Phi(\lambda) = b\,\lambda + \int_0^\infty \bigl(1 - e^{-\lambda r}\bigr)\,\nu(\diff r),
 \label{eq:bernstein_rep}
\end{equation}
where $b \ge 0$ is the drift, and $\nu$ is a L\'evy measure on $(0,\infty)$ satisfying $\int_0^\infty (1 \wedge r)\,\nu(\diff r) < \infty$.
\end{definition}

In \eqref{eq:bernstein_rep}, $b$ is the deterministic drift of the subordinator and $\nu$ is its L\'evy measure. The clock limits considered below have $b=0$.

For a driftless subordinator, the jump representation is $S_t=\int_{(0,t]\times(0,\infty)}r\,N(\diff s,\diff r)$,
where $N$ is a Poisson random measure with intensity $\diff s\,\nu(\diff r)$. For $\eta>0$, let $S_t^{(\eta)}$ retain only jumps of size greater than $\eta$. There are finitely many such jumps on a bounded time interval. The discarded mass satisfies
\begin{equation}
 \E\!\left[\sup_{t\le H}|S_t-S_t^{(\eta)}|\right]
 =H\int_{(0,\eta]}r\,\nu(\diff r)\longrightarrow0.
 \label{eq:prelim_small_jump_mass}
\end{equation}
Indeed, the difference is nonnegative and nondecreasing, so its supremum is its value at $H$, whose expectation is given by the Poisson integral formula. This estimate controls operational time; transferring it to the slow motion requires the stability argument in Section~\ref{sec:proofs}.

In particular, arising as the scaling limits of heavy-tailed rewards, the strictly $\beta$-stable subordinator $S^\beta = (S_t^\beta)_{t\ge0}$ with stability index $\beta\in(0,1)$ and scale parameter $\kappa>0$ is of central importance throughout this paper. Its Laplace transform is $\E[e^{-\lambda S_t^\beta}]=e^{-t\Phi_\beta(\lambda)}$, where the Laplace exponent is given by the Bernstein function
\begin{equation}
 \Phi_\beta(\lambda) = \kappa\lambda^\beta, \qquad \lambda > 0.
 \label{eq:stable_bernstein_exponent}
\end{equation}
By the standard integral identity $\int_0^\infty (1 - e^{-\lambda r})r^{-1-\beta}\,\diff r = \frac{\Gamma(1-\beta)}{\beta}\lambda^\beta$, the drift coefficient of $S^\beta$ vanishes ($b=0$) and its L\'evy measure on $(0,\infty)$ is given explicitly by
\begin{equation}\label{eq:stable_levy_measure_def}
 \nu_\beta(\diff r) = \frac{\kappa\,\beta}{\Gamma(1-\beta)}\,r^{-1-\beta}\,\diff r, \qquad r > 0.
\end{equation}
The dependence of $S^\beta$ and $\nu_\beta$ on $\kappa$ is implicit in the notation. Since the limit is a stable subordinator, Laplace transforms are more natural than characteristic functions, as they connect directly to the regularly varying tail and the limiting exponent \(\kappa\lambda^\beta\).
The strictly $\beta$-stable subordinator satisfies the self-similarity scaling relation $S_{ct}^\beta \stackrel{\mathrm{d}}{=} c^{1/\beta}S_t^\beta$ for all $c>0$, and its sample paths are almost surely strictly increasing pure-jump trajectories with infinite jump activity and finite total variation on bounded time intervals. Indeed, $\nu_\beta((0,1))=\infty$ whereas $\int_0^1r\,\nu_\beta(\diff r)<\infty$, and monotonicity gives $\operatorname{TV}_{[0,H]}(S^\beta)=S_H^\beta<\infty$ almost surely.

Finally, the preceding clock and semigroup constructions are linked by subordination, whose generator is given by the Phillips formula.
Let $(P_t)_{t\ge0}$ be a strongly continuous contraction semigroup with generator $\Lop$. For a Bernstein function $\Phi$ satisfying $\Phi(0+)=0$, let $\rho_t=\Law(S_t)$ be the law of the associated subordinator. The subordinated semigroup is defined by the Bochner integral
\[
 P_t^\Phi f=\int_{[0,\infty)}P_rf\,\rho_t(\diff r).
\]
The Phillips formula on $D(\Lop)$ is
\begin{equation}
 -\Phi(-\Lop)f=b\Lop f+\int_0^\infty(P_rf-f)\,\nu(\diff r).
 \label{eq:phillips_prelim}
\end{equation}
The integral converges in $\|\cdot\|_\infty$: Lemma~\ref{ass:semigroup_lipschitz} bounds the integrand by $r\|\Lop f\|_{\infty}$ for $r\le1$, while contractivity bounds it by $2\|f\|_{\infty}$ for $r>1$. The closure of the operator defined by the right-hand side on $D(\Lop)$ is the generator of $(P_t^\Phi)_{t\ge0}$; see \cite{Phillips1952,SchillingSongVondracek2012}. In particular, for the $\beta$-stable subordinator $S^\beta$ with $b=0$ and $\Phi_\beta(\lambda)=\kappa\lambda^\beta$, formula \eqref{eq:phillips_prelim} identifies the fractional power generator
\begin{equation}
 -\Phi_\beta(-\Lop)f = -\kappa(-\Lop)^\beta f = \int_0^\infty(P_rf-f)\,\nu_\beta(\diff r), \qquad f\in D(\Lop).
 \label{eq:fractional_power_phillips}
\end{equation}
This formula does not assert generation for the sum with a second, noncommuting generator.

\section{Stable Clocks from Random Environments}
\label{sec:environmental_criteria}
\label{subsec:clock_verification}

Building on classical stable-limit theory, this section establishes a self-contained environmental criterion for the homogenization problem studied here and provides explicit verification examples. The criterion is expressed in terms of integrated cycle rewards and gives the precise normalization, allows dependence between reward and duration within a cycle, handles delayed stationary initialization, and provides both activity-measure and Skorokhod $M_1$ convergence.

\subsection{Stable-clock criterion}
\label{subsec:regenerative_conditions}

We use $\cM([0,\infty))$ for nonnegative Radon measures with the vague topology and $M_1$ for the Skorokhod topology of cumulative clocks; definitions and the measure-to-clock lemma are given in Appendix~\ref{app:measure_topology}. We state the regenerative stable-clock criterion for the cumulative operational clocks and the activity measures
\[
A_t^\varepsilon:=\frac1{a_{\varepsilon^{-1}}}\int_0^{t/\varepsilon}Q(\theta_s\omega)\,ds, \qquad \mu^\varepsilon(dt):=dA_t^\varepsilon.
\]
The $\beta$-stable subordinator limits for positive infinite-mean processes and $M_1$ limits for clustered regularly varying sequences are developed in \cite{AaronsonZweimuller2014,BasrakKrizmanicSegers2012}. Melbourne and Zweim\"uller \cite[Theorem~2.2]{MelbourneZweimuller2015} lift induced functional limits to the original discrete-time system using integrable return times and control of within-excursion oscillations. Here nonnegative activity ensures monotone accumulation; we use i.i.d.\ complete cycles with regularly varying rewards and finite mean durations, allowing reward--duration dependence.

Theorem~\ref{thm:clock_limit} supplies the continuous-time regenerative criterion, its normalization, and the treatment of a finite initial segment in the form used by the homogenization theorem. Its proof uses classical regular-variation arguments \cite{Resnick1987}; vague measure convergence and $M_1$ clock convergence are linked by the standard monotone-path criterion \cite{Whitt2002}, as detailed in Appendix~\ref{app:measure_topology}.

\begin{theorem}\label{thm:clock_limit}
Under Assumptions~\ref{ass:local_activity} and~\ref{ass:level_C}, the activity measures converge vaguely to the measure of a strictly $\beta$-stable subordinator, and the cumulative operational clocks converge in the local $M_1$ topology (and hence in finite-dimensional distributions):
\begin{equation}\label{eq:regenerative_clock_limit}
 \mu^\varepsilon\Longrightarrow\diff S^\beta
 \quad\text{in }\cM([0,\infty)),\qquad
 A^\varepsilon
\Longrightarrow S^\beta
\quad\text{in }
\bigl(D([0,\infty)),M_1\bigr), \quad \text{as } \varepsilon\downarrow0,
\end{equation}
where the limiting process $S^\beta$ is the $\beta$-stable subordinator, with scale parameter $\kappa=\Gamma(1-\beta)/\mu_\tau$ determined by the environmental cycle characteristics.
\end{theorem}

\begin{proof}[Proof of Theorem~\ref{thm:clock_limit}]
We first assume a process initiated at a regeneration epoch ($\tau_0=0$), and then extend the result to the general stationary initialization with an initial incomplete cycle allowed by Assumption~\ref{ass:level_C}.

Write $T=\varepsilon^{-1}$ and $n_T(t)=\lfloor Tt/\mu_\tau\rfloor$. 
Let $B_n=\sum_{k=1}^n R_k$ is the total reward accumulated over the first $n$ complete cycles, and $N_{\mathrm{reg}}(u)=\max\{k:\tau_k\le u\}$ is the number of cycles completed by time $u$. By Assumption~\ref{ass:level_C}(a), the cycle durations $\tau_k-\tau_{k-1}$ are i.i.d.\ with finite mean $\mu_\tau$. The strong law of large numbers therefore gives $\tau_n/n\to\mu_\tau$ and hence $N_{\mathrm{reg}}(Tt)/T\to t/\mu_\tau$ almost surely for each $t>0$.

We first evaluate the Laplace transform of the normalized single-cycle reward $R_1/a_T$. 
Define $g_T(\lambda)=\E[e^{-\lambda R_1/a_T}]$ for $\lambda>0$. 
Recall that $T\Pp(R_1>a_Tx)\to x^{-\beta}$ as $T\to\infty$ pointwise for each $x>0$. Integration by parts and the regular-variation tail \eqref{eq:regular_variation} yield
\begin{align} \label{eq:reward_laplace_limit}
 T\E\!\left[1-e^{-\lambda R_1/a_T}\right] &=\int_0^\infty \lambda e^{-\lambda x}\,T\Pp(R_1>a_Tx)\,\diff x
 \notag\\
 &\longrightarrow\int_0^\infty\lambda e^{-\lambda x}x^{-\beta}\,\diff x =\Gamma(1-\beta)\lambda^\beta,
 \quad \text{as } T\to\infty.
\end{align}
Here, the regular-variation bound is $T\Pp(R_1>a_Tx)\le Cx^{-\beta-\rho}$ for $x\le1$ and $\rho \in (0,1-\beta)$ once $a_Tx$ exceeds a fixed threshold; below that threshold the integral is bounded by a constant times $T/a_T\to0$, while on $x\ge1$, $T\Pp(R_1>a_Tx)\le Cx^{-\beta+\rho}$ is integrable against $\lambda e^{-\lambda x}$. 

For any finite collection of observation times $0=t_0<t_1<\cdots<t_m$, the cycle-reward block associated with the $j$th time interval $(t_{j-1},t_j]$ is $B_{n_T(t_j)}-B_{n_T(t_{j-1})}$. These blocks involve disjoint deterministic sets of cycle indices and are therefore independent, since the rewards $R_k$ are i.i.d.\ For Laplace parameters $\lambda_j\ge0$, this independence yields
\begin{align}
 \E\exp\left\{-\sum_{j=1}^m\lambda_j
 \frac{B_{n_T(t_j)}-B_{n_T(t_{j-1})}}{a_T}\right\}
 &=\prod_{j=1}^m  \E[e^{-\lambda_j R_1/a_T}]^{n_T(t_j)-n_T(t_{j-1})}\notag\\
 &\longrightarrow
 \exp\left\{-\frac{\Gamma(1-\beta)}{\mu_\tau}
       \sum_{j=1}^m(t_j-t_{j-1})\lambda_j^\beta\right\},\quad \text{as } T\to\infty.
 \label{eq:deterministic_cycle_fdds}
\end{align}
By \eqref{eq:stable_bernstein_exponent}, this limiting expression is the joint Laplace transform of the increments of the $\beta$-stable subordinator $S^\beta$ with scale parameter $\kappa = \mu_\tau^{-1}\Gamma(1-\beta)$.

It remains to replace the deterministic cycle indices $n_T(t)$ by the random renewal counting process $N_{\mathrm{reg}}(Tt)$. Since the observable $Q \ge 0$, we have
\begin{equation}\label{eq:renewal_sandwich}
 B_{N_{\mathrm{reg}}(Tt)}\le \int_0^{Tt}Q(\theta_s\omega)\,\diff s\le B_{N_{\mathrm{reg}}(Tt)+1}.
\end{equation}
Fix $0<\delta<t/\mu_\tau$ and define the index bounds $n_- =\lfloor T(t/\mu_\tau-\delta)\rfloor$ and $n_+=\lceil T(t/\mu_\tau+\delta)\rceil$. Since $N_{\mathrm{reg}}(Tt)/T\to t/\mu_\tau$ almost surely, we have
\[ \lim_{T \to \infty}\mathbb{P} \Big( n_-\le N_{\mathrm{reg}}(Tt)\le n_+ \Big) = 1. \]
On event $\{n_-\le N_{\mathrm{reg}}(Tt)\le n_+\}$, the operational clock error is bounded by
\[
 \left|A_t^\varepsilon-\frac{B_{n_T(t)}}{a_T}\right|
 \le \frac{B_{n_++1}-B_{n_-}}{a_T}.
\]
For any $z,\lambda>0$, Markov's inequality applied to \(1-e^{-\lambda z}\) gives
\begin{align}\label{eq:boundary_band}
 \limsup_{T\to\infty}\Pp\left( \frac{B_{n_++1}-B_{n_-}}{a_T}>z\right)
 = & \limsup_{T\to\infty}\Pp\left( 1-e^{-\lambda(B_{n_++1}-B_{n_-})/a_T}>1-e^{-\lambda z}\right) \notag\\
 \le & \limsup_{T\to\infty}
 \frac{\E\left[1-e^{-\lambda(B_{n_++1}-B_{n_-})/a_T}\right]}{1-e^{-\lambda z}}
 \notag\\
 = & \limsup_{T\to\infty} \frac{1-g_T(\lambda)^{\,n_++1-n_-}}{1-e^{-\lambda z}}
 \notag\\
 = & \frac{1-e^{-2\delta\Gamma(1-\beta)\lambda^\beta}}
              {1-e^{-\lambda z}}.
\end{align}
Here $(n_++1-n_-)/T\to2\delta$, so \eqref{eq:reward_laplace_limit} yields the last equality.
Letting $\delta\downarrow0$ after $T\to\infty$ shows that, for every $z>0$,
\[
 \lim_{\delta\downarrow0}\limsup_{T\to\infty}
 \Pp\left(\frac{B_{n_++1}-B_{n_-}}{a_T}>z\right)=0.
\]
Applying the same estimate at the finitely many observation times and combining with \eqref{eq:deterministic_cycle_fdds} via Slutsky's theorem proves the claimed finite-dimensional clock limit: $A^\varepsilon\xrightarrow{\mathrm{f.d.d.}}S^\beta$. 

Next we establish vague convergence of the activity measures $\mu^\varepsilon$ in $\cM([0,\infty))$. For a test function $\varphi\in C_c([0,\infty))$ with $\text{supp}\varphi \subset [0,H]$, approximate $\varphi$ uniformly on $[0,H]$ by step functions 
\[ \varphi_m(t) = \varphi(0)\mathbf 1_{\{0\}}(t)+\sum_{j=1}^m c_j \mathbf 1_{(t_{j-1},t_j]}(t), \qquad  0=t_0<t_1<\cdots<t_m=H \]
with deterministic endpoints, at which the limiting clock is almost surely continuous. 

Since $A^\varepsilon\xrightarrow{\mathrm{f.d.d.}}S^\beta$, we have
\[ 
\int\varphi_m\,d\mu^\varepsilon = \sum_{j=1}^m c_j \bigl(A_{t_j}^\varepsilon-A_{t_{j-1}}^\varepsilon\bigr) \Longrightarrow \int\varphi_m\,dS^\beta, \quad \text{as } \varepsilon \to 0. 
\]
Here, the approximation error is bounded:
\[
\left| \int(\varphi-\varphi_m)\,d\mu^\varepsilon \right| \le \|\varphi-\varphi_m\|_{\infty} \mu^\varepsilon([0,H]) = \|\varphi-\varphi_m\|_{\infty} A_H^\varepsilon.
\]
Thus these total masses are tight by $A_H^\varepsilon\Longrightarrow S_H^\beta$ at the endpoint $t = H$. It follows that 
\[
\int^H _0 \varphi\,d\mu^\varepsilon \Longrightarrow \int^H _0 \varphi\,dS^\beta \quad \text{as }\varepsilon \to 0.
\]
The same argument applies jointly to every finite collection of test functions. Tightness of the masses on every $[0,H]$ gives tightness in the vague topology: choose bounds on the masses on $[0,n]$, $n\in\N$, with summable failure probabilities and use relative compactness of locally mass-bounded measures. The test integrals identify every subsequential limit as $\diff S^\beta$. This proves the measure convergence in \eqref{eq:regenerative_clock_limit} for a process started at $\tau_0=0$.

Finally, consider the general stationary initialization with an initial incomplete cycle of duration $\tau_0 <\infty$ almost surely. By Assumption~\ref{ass:local_activity}, whenever $\tau_0<\infty$, we can choose an integer $n> \tau_0$, so that
\[
 R_{\mathrm{init}}:=\int_0^{\tau_0} Q(\theta_s\omega)\,\diff s
 \le\int_0^n Q(\theta_s\omega)\,\diff s<\infty, \quad \text{almost surely}. 
\]
Since $\tau_0<\infty$ almost surely, $\tau_0/T\to0$ almost surely as $T\to\infty$. Furthermore, $R_{\mathrm{init}}/a_T\to0$ almost surely since $a_T\to\infty$. Thus for fixed $t>0$ and $\delta\in(0,t)$, on $\{\tau_0/T\le\delta\}$ the shifted post-regeneration clock at $t$ lies between its clocks at $t-\delta$ and $t$. 
The latter increments converge in law to $S_t^\beta-S_{t-\delta}^\beta$, which tends to zero in probability as $\delta\downarrow0$. The initial reward divided by $a_T$ tends to zero almost surely. This proves the same finite-dimensional limits after the delay, and the preceding test-function argument gives the measure convergence $\mu^\varepsilon\Longrightarrow\diff S^\beta$. 
Since $S^\beta$ has no fixed atoms, Lemma~\ref{lem:measure_clock_continuity}(b) now gives convergence of the cumulative clocks in the local $M_1$ topology and on each deterministic interval $[0,H]$. Thus Theorem~\ref{thm:clock_limit} holds with the stated initialization.
\end{proof}

\begin{remark}[Microscopic origin of the stable clock]
\label{rem:burst_interpretation}
On the scale $T=\varepsilon^{-1}$, order $T$ regeneration cycles occur per unit macroscopic time, whereas only order one have integrated rewards comparable to $a_T$. These rare episodes produce the macroscopic activity atoms. Their physical durations vanish: writing $\ell_k=\tau_k-\tau_{k-1}$, finite mean duration gives, for every $\delta>0$,
\begin{equation}
 T\Pp(\ell_1>\delta T)
 \le\delta^{-1}\E[\ell_1\ind_{\{\ell_1>\delta T\}}]\longrightarrow0.
 \label{eq:cycle_duration_bound}
\end{equation}
A union bound therefore controls the maximal duration among order $T$ cycles. The heavy tail of the integrated reward determines the atom sizes, while the renewal rate $\mu_\tau^{-1}$ determines their frequency; see Figure~\ref{fig:regenerative_burst_clock}. Thus absolutely continuous microscopic activity concentrates into the atomic measure $\diff S^\beta$, while the increasing clocks converge in Skorokhod $M_1$ topology.
\end{remark}

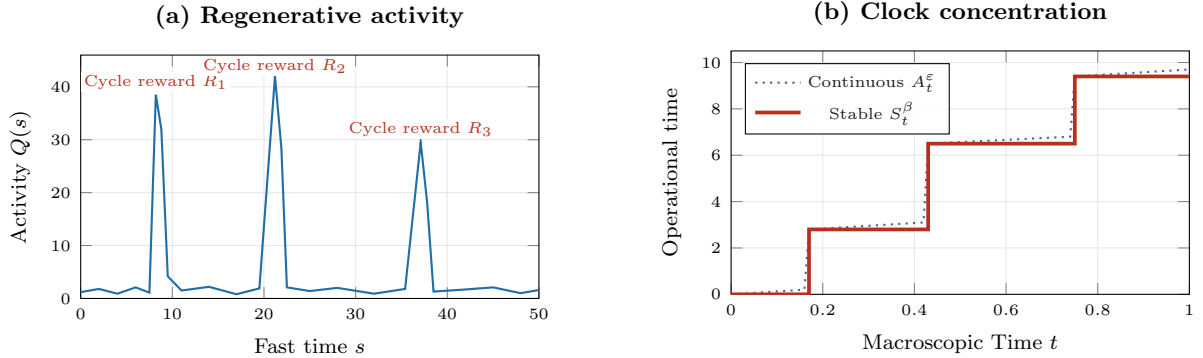
\begin{figure}[htbp]
\centering
\begin{tikzpicture}
\begin{axis}[
    width=0.48\textwidth,
    height=4.8cm,
    title={\textbf{(a) Regenerative activity}},
    title style={font=\footnotesize},
    xlabel={Fast time $s$},
    ylabel={Activity $Q(s)$},
    xlabel style={font=\scriptsize},
    ylabel style={font=\scriptsize},
    tick label style={font=\tiny},
    xmin=0, xmax=50,
    ymin=0, ymax=46,
    xtick={0, 10, 20, 30, 40, 50},
    ytick={0, 10, 20, 30, 40},
    grid=both,
    grid style={line width=.1pt, draw=gray!20}
]
\addplot[line width=0.8pt, RoyalBlue!80!black] coordinates {
    (0,1.2) (2,1.8) (4,0.9) (6,2.1) (7.5,1.1) (8.2,38.5) (8.8,32.0) (9.5,4.2)
    (11,1.5) (14,2.2) (17,0.8) (19.5,1.9) (21.2,42.0) (21.9,28.4) (22.5,2.1)
    (25,1.4) (28,2.0) (32,0.9) (35.4,1.8) (37.1,29.5) (37.8,18.2) (38.5,1.3)
    (42,1.7) (45,2.1) (48,1.0) (50,1.6)
};
\node[font=\tiny, BrickRed, fill=white, inner sep=1pt] at (axis cs: 8.2, 41) {Cycle reward $R_1$};
\node[font=\tiny, BrickRed, fill=white, inner sep=1pt] at (axis cs: 21.2, 44) {Cycle reward $R_2$};
\node[font=\tiny, BrickRed, fill=white, inner sep=1pt] at (axis cs: 37.1, 32) {Cycle reward $R_3$};
\end{axis}
\end{tikzpicture}
\hfill
\begin{tikzpicture}
\begin{axis}[
    width=0.48\textwidth,
    height=4.8cm,
    title={\textbf{(b) Clock concentration}},
    title style={font=\footnotesize},
    xlabel={Macroscopic Time $t$},
    ylabel={Operational time},
    xlabel style={font=\scriptsize},
    ylabel style={font=\scriptsize},
    tick label style={font=\tiny},
    xmin=0, xmax=1.0,
    ymin=0, ymax=10.5,
    xtick={0.0, 0.2, 0.4, 0.6, 0.8, 1.0},
    ytick={0, 2, 4, 6, 8, 10},
    grid=both,
    grid style={line width=.1pt, draw=gray!20},
    legend style={font=\tiny, at={(0.03,0.95)}, anchor=north west, fill opacity=0.9}
]
\addplot[line width=0.8pt, RoyalBlue!60!black, dotted] coordinates {
    (0.00,0.0) (0.16,0.2) (0.17,2.8) (0.42,3.1) (0.43,6.5) (0.74,6.8) (0.75,9.4) (1.00,9.7)
};
\addlegendentry{Continuous $A_t^\varepsilon$}

\addplot[line width=1.4pt, BrickRed] coordinates {
    (0.00,0.0) (0.17,0.0) (0.17,2.8) (0.43,2.8) (0.43,6.5) (0.75,6.5) (0.75,9.4) (1.00,9.4)
};
\addlegendentry{Stable $S_t^\beta$}
\end{axis}
\end{tikzpicture}

\caption{Illustration of the regenerative burst mechanism and clock concentration. 
The areas accumulated during prominent activity cycles correspond to large rewards $R_k$ (left); after rescaling, their short time spans produce jumps of the limiting operational clock (right). The plotted curves are illustrative, not sample paths from a specified model.}
\label{fig:regenerative_burst_clock}
\end{figure}

\subsection{Examples: resetting pulses and an ergodic diffusion}
\label{sec:verification}

We verify Assumptions~\ref{ass:local_activity} and~\ref{ass:level_C} for two complementary environments. Poisson resetting pulses give arbitrary stable indices $\beta\in(0,1)$ through heavy-tailed external amplitudes and finite-mean cycle durations. A stationary Ornstein--Uhlenbeck diffusion, with a nonlinear activity observable, gives an explicit $1/2$-stable cycle reward. The latter environment will drive the stochastic-volatility application in Section~\ref{subsec:ou_volatility}.

\paragraph{Model 1: Poisson resetting pulses with exponential decay.}
This environment describes external shocks that replace the current activity by a new random amplitude, followed by exponential relaxation. It is a resetting mechanism, rather than additive shot noise: the latest shock replaces the residual activity of earlier shocks. Its explicit renewal structure provides a tractable description of intermittent excitation, transient stress, or event-driven volatility; the Pareto amplitudes below are a modeling assumption.
Let $N=\sum_{j\in\Z}\delta_{T_j}$ be a stationary Poisson point process on $\R$ with intensity $\lambda>0$, indexed so that $T_0\le0<T_1$. Independently of $N$, let $(E_j)_{j\in\Z}$ be i.i.d.\ marks with
$\Pp(E_0>x)=x^{-\beta}$ for $x\ge1$, where $0<\beta<1$. These Pareto amplitudes have lower cutoff $1$ and tail index $\beta$: their power-law tail allows occasional large shocks, and $\beta<1$ implies an infinite mean. For $\gamma>0$, define
\begin{equation}
 Q(t)=E_j e^{-\gamma(t-T_j)},\qquad T_j\le t<T_{j+1}.
 \label{eq:shot_noise_burst}
\end{equation}
This resetting mechanism is particularly suited to settings involving changes in operating conditions, preparation of a fresh environment, or the termination of one activity episode followed by the start of another.
At each arrival, the activity resets to the new mark. This reset makes the subsequent cycle independent of the preceding cycles. {
On the marked configuration space, write $\omega=\sum_{j\in\Z}\delta_{(T_j,E_j)}$ and define time translation by
\[
 \theta_t\omega=\sum_{j\in\Z}\delta_{(T_j-t,E_j)},\qquad t\in\R.
\]
Let $F_E$ be the mark distribution. The configuration law is preserved by this shift because its intensity $\lambda\,\diff s\,F_E(\diff e)$ is translation invariant. Events determined by bounded time windows become independent when their windows are shifted far enough apart. Approximation by such events gives mixing, hence ergodicity, of the translation flow. If $Q(\omega)$ denotes the activity at time $0$, then $Q(\theta_t\omega)=Q(t)$ from \eqref{eq:shot_noise_burst}. The activity process therefore inherits stationarity and ergodicity as a measurable function of this flow.
}

Take $\tau_0=T_1$ and $\tau_k=T_{k+1}$ for $k\ge1$. The complete cycle durations
$\ell_k=T_{k+1}-T_k$ are i.i.d.\ exponential random variables of rate $\lambda$, independent of the marks, with $\mu_\tau=\lambda^{-1}$. The complete cycle segments are functions of the i.i.d.\ pairs $(E_k,\ell_k)$ and are independent of the initial segment. Moreover, $\tau_0<\infty$ almost surely. A pulse with initial amplitude $E_j$ contributes at most the area obtained by allowing it to decay indefinitely:
\[
 \int_0^\infty E_j e^{-\gamma s}\,\diff s=\frac{E_j}{\gamma}.
\]
Since $T_0\le0<T_1$, the initial segment $[0,T_1)$ is a portion of the pulse started at $T_0$. Since each Pareto mark is finite almost surely, extending its exponential decay beyond $T_1$ therefore gives
\[
 R_{\mathrm{init}}:=\int_0^{T_1}Q(s)\,\diff s
\le\int_{T_0}^\infty E_0e^{-\gamma(s-T_0)}\,\diff s
 =\frac{E_0}{\gamma}<\infty
 \quad\text{almost surely}.
\]
Similarly, summing these pulse-area bounds over the arrivals on $[0,H]$ yields
\[
 \int_0^H Q(s)\,\diff s
 \le\frac1\gamma\left(E_0+\sum_{j:\,0<T_j\le H}E_j\right)<\infty
 \quad\text{almost surely}.
\]
The last sum is finite because there are finitely many arrivals on $[0,H]$. Thus the local-finiteness and initialization requirements hold.

For each complete cycle,
\[
 R_k=E_kV_k,\qquad
 V_k=\frac{1-e^{-\gamma\ell_k}}{\gamma}\in(0,\gamma^{-1}).
\]
Although $R_k$ and $\ell_k$ are dependent, $E_k$ and $V_k$ are independent. Conditioning on $V_1$ and using the exact Pareto tail gives, for $x\ge\gamma^{-1}$,
\begin{equation}
 \Pp(R_1>x)
 =\E\!\left[\Pp(E_1>x/V_1\mid V_1)\right]
 =C_\beta x^{-\beta},
 \qquad C_\beta=\E[V_1^\beta]\in(0,\infty).
 \label{eq:breiman_shot_noise}
\end{equation}
This is the bounded-factor Pareto case of the product-tail principle in \cite{Breiman1965}. Explicitly,
\[
 C_\beta=\int_0^\infty
 \left(\frac{1-e^{-\gamma s}}{\gamma}\right)^\beta
 \lambda e^{-\lambda s}\,\diff s
 =\gamma^{-\beta}
 \frac{\Gamma(1+\lambda/\gamma)\Gamma(1+\beta)}
      {\Gamma(1+\lambda/\gamma+\beta)}.
\]
Consequently, Assumption~\ref{ass:level_C}(b) holds and one may take
$a_T=(C_\beta T)^{1/\beta}$. Theorem~\ref{thm:clock_limit} yields
$\mu^\varepsilon\Rightarrow\diff S^\beta$ and
$A^\varepsilon\mathrel{\xRightarrow{M_1}}S^\beta$, now with
$\kappa=\lambda\Gamma(1-\beta)$. The factor $C_\beta$, which contains the effect of the decay rate $\gamma$, is absorbed into the reward normalization $a_T$.

{
\paragraph{Model 2: An Ornstein--Uhlenbeck diffusion environment.}
We consider the random environment given by a stationary Ornstein--Uhlenbeck diffusion. The stable additive-functional limit in this example falls within the theory of B\'ethencourt \cite{Bethencourt2023}.

Let $\vartheta>0$ and let $Y$ be the two-sided stationary solution of
\begin{equation}
 \diff Y_s=-\vartheta Y_s\,\diff s+\sqrt{2\vartheta}\,\diff B_s^Y,
 \qquad
 Y_s=\sqrt{2\vartheta}\int_{-\infty}^s
       e^{-\vartheta(s-u)}\,\diff B_u^Y.
 \label{eq:ou_environment}
\end{equation}
Its invariant law is $\pi(\diff y)=(2\pi)^{-1/2}e^{-y^2/2}\,\diff y$, and
$\E[Y_sY_t]=e^{-\vartheta|t-s|}$. On $\Omega=C(\R;\R)$, write $Y_s(\omega)=\omega(s)$ and define
\[
 (\theta_t\omega)(s)=\omega(s+t),\qquad s,t\in\R.
\]
The stationary OU law on this path space is preserved by $\theta_t$. The decaying covariance makes Gaussian cylinder events asymptotically independent under widely separated shifts; approximation by cylinder events then gives mixing, hence ergodicity. Define
\begin{equation}
 Q(\omega)=e^{\omega(0)^2},\qquad Q(\theta_s\omega)=e^{Y_s^2}.
 \label{eq:ou_activity}
\end{equation}
Although $\E_\pi[Q]=\infty$, continuity of $Y$ gives, for every finite $H$,
\[
 \int_0^H e^{Y_s^2}\,\diff s
 \le H\exp\!\left(\sup_{0\le s\le H}|Y_s|^2\right)<\infty
 \qquad\text{almost surely}.
\]
Thus Assumption~\ref{ass:local_activity} holds. The boundedness requirements in Assumption~\ref{ass:slow_coefficients} concern the slow coefficients, not the environmental observable or the OU drift.

Fix $h>0$ and define complete cycles by successive passages $0\to h\to0$:
\begin{equation}
 \begin{split}
 \tau_0&=\inf\{s\ge0:Y_s=0\},\\
 \rho_k&=\inf\{s>\tau_{k-1}:Y_s=h\},\qquad
 \tau_k=\inf\{s>\rho_k:Y_s=0\},\quad k\ge1,\\
 \ell_k&=\tau_k-\tau_{k-1},\qquad
 R_k=\int_{\tau_{k-1}}^{\tau_k}e^{Y_s^2}\,\diff s.
 \end{split}
 \label{eq:ou_regeneration}
\end{equation}
Recurrence of the OU diffusion makes these times finite almost surely. The strong Markov property at the returns to $0$ makes the complete cycle segments i.i.d.\ and independent of the initial segment. Also,
$R_{\mathrm{init}}=\int_0^{\tau_0}e^{Y_s^2}\,\diff s<\infty$ almost surely by continuity on the finite initial interval.

Set
\[
 s(y)=\int_0^y e^{z^2/2}\,\diff z,\qquad s_h=s(h).
\]
For $H_a=\inf\{t\ge0:Y_t=a\}$, the hitting-time equations for
$\mathcal A_Y=\vartheta(\partial_{yy}-y\partial_y)$ give
\begin{align}
 \E_0H_h&=\frac1\vartheta\int_0^h e^{y^2/2}
                    \int_{-\infty}^y e^{-z^2/2}\,\diff z\,\diff y,\notag\\
 \E_hH_0&=\frac1\vartheta\int_0^h e^{y^2/2}
                    \int_y^\infty e^{-z^2/2}\,\diff z\,\diff y,\notag\\
 \mu_h:=\E\ell_1&=\E_0H_h+\E_hH_0
                =\frac{\sqrt{2\pi}}{\vartheta}s_h\in(0,\infty).
 \label{eq:ou_mean_cycle}
\end{align}
In particular, the mean cycle-duration requirement in Assumption~\ref{ass:level_C}(a) is satisfied.

The choice of activity in \eqref{eq:ou_activity} permits an exact reward calculation. Since $\mathcal A_Ys=0$, It\^o's formula, for a cycle started at $Y_0=0$, gives
\begin{equation}
 \diff s(Y_t)=\sqrt{2\vartheta}\,e^{Y_t^2/2}\,\diff B_t^Y,
 \qquad
 \langle s(Y)\rangle_t=2\vartheta\int_0^t e^{Y_u^2}\,\diff u.
 \label{eq:ou_scale_clock}
\end{equation}
This quadratic-variation clock is continuous, strictly increasing, and unbounded. The Dambis--Dubins--Schwarz representation therefore turns the passages $0\to h\to0$ into Brownian passages $0\to s_h\to0$. By the Brownian strong Markov property,
\[
 2\vartheta R_1\overset{\mathrm{law}}{=}H_{s_h}^B+\widetilde H_{s_h}^B,
\]
where the right-hand terms are independent copies of the first time a standard Brownian motion started at $0$ reaches $s_h$. Independence of the Brownian motion in the time-change representation from its clock is neither asserted nor needed. The first-passage transform
$\E[e^{-qH_{s_h}^B}]=e^{-s_h\sqrt{2q}}$ now yields
\begin{equation}
 \E[e^{-zR_1}]=\exp\!\left(-\frac{2s_h}{\sqrt\vartheta}\sqrt z\right),
 \qquad z\ge0.
 \label{eq:ou_reward_transform}
\end{equation}
Thus the complete-cycle reward itself has a positive $1/2$-stable law, with
\begin{equation}
 f_R(r)=\frac{s_h}{\sqrt{\pi\vartheta}}r^{-3/2}
             e^{-s_h^2/(\vartheta r)},\qquad
 \Pp(R_1>r)\sim C_h r^{-1/2},\qquad
 C_h=\frac{2s_h}{\sqrt{\pi\vartheta}}.
 \label{eq:ou_reward_tail}
\end{equation}
Consequently, Assumption~\ref{ass:level_C}(b) holds with $\beta=1/2$ and $a_T=(C_hT)^2$. Theorem~\ref{thm:clock_limit} gives
\begin{equation}
 \mu^\varepsilon\Longrightarrow\diff S,\qquad
 A^\varepsilon\xRightarrow{M_1}S,\qquad
 \E[e^{-zS_t}]=e^{-t\kappa_h\sqrt z},\qquad
 \kappa_h=\frac{\sqrt\pi}{\mu_h}=\frac{\vartheta}{\sqrt2\,s_h}.
 \label{eq:ou_clock_limit}
\end{equation}
Here $\beta=1/2$, $\mu_\tau=\mu_h$, $\kappa=\kappa_h$, and $S=S^{1/2}$. The corresponding rate and clock are
\begin{equation}
 r_t^\varepsilon=\frac{\varepsilon}{C_h^2}e^{Y_{t/\varepsilon}^2},
 \qquad
 A_t^\varepsilon=\frac{\varepsilon^2}{C_h^2}
                      \int_0^{t/\varepsilon}e^{Y_s^2}\,\diff s.
 \label{eq:ou_activity_scaling}
\end{equation}
The auxiliary level $h$ fixes the reward normalization. Equivalently, the unnormalized rate $\widehat r_t^\varepsilon=\varepsilon e^{Y_{t/\varepsilon}^2}$ gives $\widehat A^\varepsilon=C_h^2A^\varepsilon$ and a limiting clock with Laplace exponent $\sqrt{2\vartheta/\pi}\sqrt z$, independent of $h$.

This additive-functional limit is also a specialization of B\'ethencourt's diffusion theorem \cite{Bethencourt2023}. In his notation, take $\sigma(y)=\sqrt{2\vartheta}$, $f(y)=e^{y^2}$, stable index $1/2$, and slowly varying function identically one. The power of $|s(y)|$ in his scale-function condition is then zero, and
\[
 [\sigma(y)s'(y)]^{-2}f(y)
 =\frac{1}{2\vartheta}.
\]
After removal of the finite initial segment, that theorem gives the finite-dimensional limit of $\widehat A^\varepsilon$ with the same normalization and Laplace exponent. Nondecreasing paths allow the usual upgrade to $M_1$ convergence \cite{Whitt2002}, with the corresponding measure convergence obtained as above. The regeneration calculation here also identifies the exact complete-cycle transform \eqref{eq:ou_reward_transform} and its tail constant \eqref{eq:ou_reward_tail}.

The stationary marginal and complete-cycle tails have different slowly varying factors:
\[
 \Pp(e^{Y_0^2}>r)\sim\sqrt{\frac2\pi}\,
                  r^{-1/2}(\log r)^{-1/2},
 \qquad
 \Pp(R_1>r)\sim C_h r^{-1/2}.
\]
The first equivalence follows from the Gaussian tail at $\sqrt{\log r}$. It cannot replace the integrated-reward calculation in \eqref{eq:ou_reward_tail}.

Both environments satisfy the environmental hypotheses of Theorem~\ref{thm:full_homogenization}, with $\kappa=\lambda\Gamma(1-\beta)$ for resetting pulses and $\kappa=\kappa_h$ for the OU environment under their respective reward normalizations. The first produces arbitrary $\beta\in(0,1)$ from external reset amplitudes; the second produces $\beta=1/2$ from excursions of a continuous Gaussian environment.
}

\section{Heavy-tailed homogenization and effective dynamics}
\label{sec:model}
\label{sec:assumptions}
\label{sec:main_theorems}

We first establish abstract semigroup homogenization given the stable environmental clock. When the activity diffusion is uniformly elliptic, the second main theorem identifies the limiting dynamics through a concrete stable-like jump kernel. We also give the effective evolution equation and compare the result with finite-mean averaging.

\subsection{Assumptions and abstract semigroup homogenization}
\label{sec:consequences}

Recall the microscopic SDE \eqref{eq:micro_sde}: the background diffusion evolves in physical time, while the activity diffusion is accelerated at the rate
\[
 r_t^\varepsilon
 =\frac{Q(\theta_{t/\varepsilon}\omega)}{\varepsilon a_{\varepsilon^{-1}}}.
\]
Its accumulated activity is $A^\varepsilon$ from \eqref{eq:cumulative_clock}. The environmental assumptions in Section~\ref{subsec:random_environment} yield its stable limit through Theorem~\ref{thm:clock_limit}; we need only supplement these environmental assumptions with conditions on the slow coefficients and the diffusion noises.

\begin{assumption}[Bounded Lipschitz coefficients]
\label{ass:slow_coefficients}
Let $b_i:\R^d\to\R^d$ and $\sigma_i:\R^d\to\R^{d\times d_i}$ for $i=0,1$.
We assume that the functions $b_0,b_1,\sigma_0,\sigma_1$ are globally bounded and globally Lipschitz. The two Brownian motions have dimensions $d_0,d_1$ and are mutually independent.
\end{assumption}
Set $a_i:=\sigma_i\sigma_i^\top$. The background and activity operators are
\begin{equation}
 \Lop_i f=b_i\cdot\nabla f+\tfrac12 a_i:D^2f,
 \qquad i=0,1.
 \label{eq:elliptic_activity}
\end{equation}
The autonomous diffusions with coefficients $b_i,\sigma_i$ generate conservative Feller semigroups $P^{(i)}$ on $C_0(\R^d)$. It\^o's formula gives $\D=C_c^\infty(\R^d)\subset D(\Lop_i)$ with the actions $\Lop_0,\Lop_1$. The semigroups $P^{(0)}$ and $P^{(1)}$ are not assumed to commute.

\begin{assumption}[Independence of the noises]
\label{ass:decoupling}
The entire environmental process is independent of $(W^0,W^1)$. This makes the conditional microscopic equation an SDE with a prescribed, locally integrable activity rate.
\end{assumption}

Under Assumptions~\ref{ass:slow_coefficients}, \ref{ass:local_activity}, \ref{ass:decoupling}, and~\ref{ass:level_C}, the microscopic SDE \eqref{eq:micro_sde} has a unique nonexplosive strong solution with continuous paths for every deterministic initial state $x\in\R^d$.
Define the extended-time  parametrization of the microscopic path by
\[
 F^\varepsilon(t)=t+A_t^\varepsilon,\qquad
 \tau^\varepsilon(u)=\inf\{t\ge0:F^\varepsilon(t)>u\},\qquad
 Z_u^{\varepsilon,x}=X_{\tau^\varepsilon(u)}^{\varepsilon,x}.
\]
\begin{figure}[H]
\centering
\begin{tikzpicture}[font=\scriptsize,>=Stealth]
\begin{scope}[x=1.25cm,y=0.85cm]
 \node[font=\small] at (1.6,4.85) {(a) Physical to extended time};
 \fill[Goldenrod!12] (1.45,0) rectangle (1.65,4.35);
 \draw[->] (0,0)--(3.15,0) node[right] {$t$};
 \draw[->] (0,0)--(0,4.5) node[left] {$u$};
 \draw[gray,densely dotted] (0,1.55)--(1.55,1.55);
 \draw[gray,densely dotted] (0,2.95)--(1.55,2.95);
 \node[left] at (0,1.55) {$u_-$};
 \node[left] at (0,2.95) {$u_+$};
 \draw[BrickRed,dashed,line width=0.9pt]
   (0,0)--(1.55,1.55)--(1.55,2.95)--(2.75,4.15);
 \draw[RoyalBlue,line width=1.2pt]
   (0,0)--(1.45,1.45)--(1.65,3.05)--(2.75,4.15);
 \draw (1.55,0.06)--(1.55,-0.06) node[below] {$t_0$};
 \node[RoyalBlue,anchor=west] at (1.95,2.8) {$F^\varepsilon(t)$};
 \node[BrickRed,anchor=west] at (1.95,2.15) {$F(t)=t+S_t^\beta$};
 \draw[decorate,decoration={brace,amplitude=3pt},Goldenrod!70!black]
   (1.65,-0.6)--(1.45,-0.6);
 \node[below,align=center] at (1.55,-0.7) {Short physical burst};
\end{scope}
\begin{scope}[shift={(7.1cm,0)},x=1.1cm,y=1.15cm]
 \node[font=\small] at (2.05,3.59) {(b) Extended to physical time};
 \fill[Goldenrod!12] (1.55,0) rectangle (2.95,3.15);
 \draw[->] (0,0)--(4.55,0) node[right] {$u$};
 \draw[->] (0,0)--(0,3.3) node[left] {$t$};
 \draw[gray,densely dotted] (0,1.55)--(1.55,1.55);
 \node[left] at (0,1.55) {$t_0$};
 \draw[gray,densely dotted] (1.55,0)--(1.55,1.55);
 \draw[gray,densely dotted] (2.95,0)--(2.95,1.55);
 \node[below] at (1.55,0) {$u_-$};
 \node[below] at (2.95,0) {$u_+$};
 \draw[RoyalBlue,line width=1.2pt]
   (0,0)--(1.45,1.45)--(3.05,1.65)--(4.15,2.75);
 \draw[BrickRed,dashed,line width=0.9pt]
   (0,0)--(1.55,1.55)--(2.95,1.55)--(4.15,2.75);
 \node[RoyalBlue] at (2.2,2.15) {$\tau^\varepsilon(u)$};
 \node[BrickRed,align=center] at (2.25,0.8) {$\tau(u)=t_0$\\Physical time frozen};
 \draw[decorate,decoration={brace,amplitude=3pt},Goldenrod!70!black]
   (2.95,-0.45)--(1.55,-0.45);
 \node[below,align=center] at (2.25,-0.52) {Expanded activity interval};
\end{scope}
\end{tikzpicture}
\caption{Schematic of the extended-time parametrization around a single burst.
Solid blue curves show the continuous microscopic clock and its inverse;
dashed red curves show the limiting clock's completed graph and its inverse.
Here $u_-=t_0+S_{t_0-}^\beta$ and $u_+=t_0+S_{t_0}^\beta$.
On $[u_-,u_+]$, physical time is frozen while $Z^x$ follows the activity
diffusion for duration $u_+-u_-=\Delta S_{t_0}^\beta$.
Microscopically, $Z_u^{\varepsilon,x}=X_{\tau^\varepsilon(u)}^{\varepsilon,x}$
traverses the same spatial path with a new time parameter.}
\label{fig:extended_time_parametrization}
\end{figure}
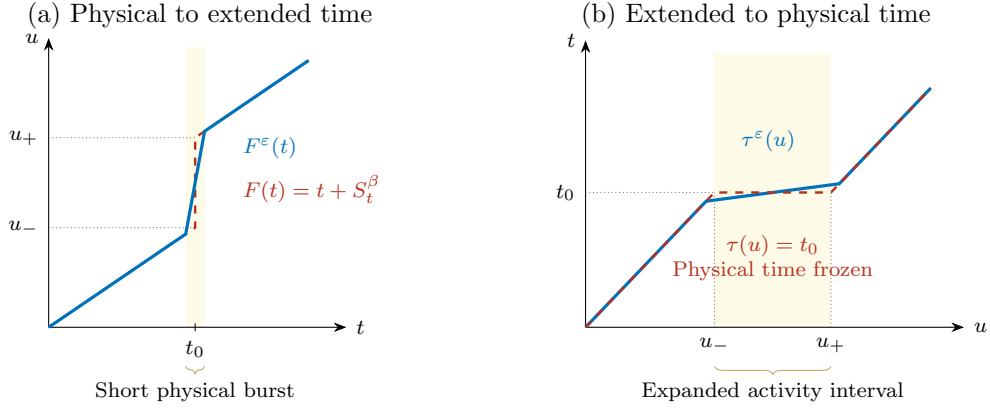
Let $\tau$ be the inverse clock in \eqref{eq:extended_clocks} associated with $S^\beta$, and let $Z^x$ solve
\begin{equation}\label{eq:extended_sde}
    Z_u^{x}=x +\int_0^u b_0(Z_v^{x})\,\diff\tau(v)  +\int_0^u \sigma_0(Z_v^{x})\,\diff B^0_{\tau(v)}
     +\int_0^u b_1(Z_v^{x})\,\diff c(v)
     +\int_0^u \sigma_1(Z_v^{x})\,\diff B^1_{c(v)} .
\end{equation}
with Brownian motions independent of $S^\beta$. 
Define the candidate physical-time limit by $\bar X_t^x=Z_{t+S_t^\beta}^x$. We equip $C_{\mathrm{loc}}([0,\infty);\R_+\times\R^d)$ with uniform convergence on bounded extended-time intervals. Set $P_t^\varepsilon f(x)=\E[f(X_t^{\varepsilon,x})]$, where the expectation averages over the environment and the diffusion noises. The parameters $\beta$, $\nu_\beta$, and $\kappa=\Gamma(1-\beta)/\mu_\tau$ are supplied by Theorem~\ref{thm:clock_limit}; fractional powers on $\D$ are defined by the Phillips formula.

{
Let $V^x$ be the autonomous activity diffusion with generator $\Lop_1$:
\[
 V_r^x=x+\int_0^r b_1(V_s^x)\,\diff s
          +\int_0^r\sigma_1(V_s^x)\,\diff W_s^1,
 \qquad r\ge0.
\]
Assumption~\ref{ass:slow_coefficients} guarantees a unique nonexplosive strong solution, with transition probabilities and semigroup $P_r^{(1)}f(x)=\E[f(V_r^x)], f\in C_0(\R^d).$
}
An activity atom of mass $r$ induces the diffusion transition $P_r^{(1)}$, which is combined in time order with the background evolution; this mechanism underlies the effective generator in the following theorem.

\begin{theorem}[Homogenization limit]
\label{thm:full_homogenization}
Under Assumptions~\ref{ass:slow_coefficients}, \ref{ass:local_activity}, \ref{ass:decoupling}, and~\ref{ass:level_C}, for every $x\in\R^d$, the process $\bar X^x$ defined above is a Markov process with Feller semigroup $\bar P$ and generator $\overline{\Lop}$. As $\varepsilon\downarrow0$:
\begin{enumerate}[label=(\roman*)]
\item Under the annealed law $\Pp\otimes\Pp_W$,
\begin{equation}\label{eq:main_fdd_limit}
 X^{\varepsilon,x}\xrightarrow{\mathrm{f.d.d.}}\bar X^x.
\end{equation}
\item For every $t\ge0$ and $f\in C_0(\R^d)$,
\begin{equation} \label{eq:main_operator_limits}
 \|P_t^\varepsilon f-\bar P_tf\|_\infty\longrightarrow0.
\end{equation}
\item The generator satisfies $C_c^\infty(\R^d)\subset D(\overline{\Lop})$ and, for every $f$ in this class,
\begin{equation}\label{eq:full_effective_gen}
 \overline{\Lop}f  =\Lop_0f-\kappa(-\Lop_1)^\beta f.
\end{equation}
\end{enumerate}
\end{theorem}

\begin{remark}
The operators $P_t^\varepsilon$ need not form a semigroup, so their Laplace-time integrals are not asserted to be resolvents; the limiting Laplace transform is the resolvent of $\overline{\Lop}$. Furthermore, \eqref{eq:main_operator_limits}
$\|P_t^\varepsilon f-\bar P_tf\|_{\infty}\le2\|f\|_{\infty}$ yield the Laplace-time convergence,
\begin{equation}
 \left\|\int_0^\infty e^{-\lambda t}P_t^\varepsilon f\,\diff t
       -(\lambda I-\overline{\Lop})^{-1}f\right\|_{\infty}
 \le\int_0^\infty e^{-\lambda t}
          \|P_t^\varepsilon f-\bar P_tf\|_{\infty}\,\diff t\longrightarrow0,
 \label{eq:proved_resolvent_convergence}
\end{equation}
for every $\lambda>0$ and $f\in C_0(\R^d)$.
\end{remark}

\begin{remark}[Path-space convergence]
\label{sec:topology}
The proof of Theorem~\ref{thm:full_homogenization} also gives extended-time functional convergence:
\begin{equation}
 (\tau^\varepsilon,Z^{\varepsilon,x})\Longrightarrow(\tau,Z^x)
 \quad\text{in }C_{\mathrm{loc}}([0,\infty);\R_+\times\R^d).
 \label{eq:main_extended_path_limit}
\end{equation}
However, this does not assert $J_1$ or $M_1$ convergence of the ordinary physical-time paths. To see this, consider the regenerative Brownian motion specialization $\Lop_0=0$, $\Lop_1=\tfrac12\Delta$, a burst with operational duration bounded away from zero occurs with probability bounded away from zero, while its physical duration vanishes. Conditional on such a burst, a Brownian coordinate has a fixed positive probability of making an internal excursion a fixed distance outside the segment joining two surrounding values, so the $M_1$ oscillation modulus remains nonzero on arbitrarily short physical-time intervals. This violates the $M_1$ tightness criterion \cite{Whitt2002} and consequently rules out both $M_1$ and $J_1$ convergence of the physical-time paths. This latter obstruction concerns the stated Brownian realization, not every activity semigroup. The nondecreasing clocks themselves do converge in $M_1$; their monotone-path criterion does not extend to the slow state. The parametrized limit retains the internal diffusion trajectory: on an interval $[t+S_{t-}^\beta,t+S_t^\beta]$ corresponding to a clock jump, $\tau$ is constant while $Z^x$ evolves through the activity diffusion. Thus \eqref{eq:main_extended_path_limit} is compatible with the physical-time $t$ obstruction.
\end{remark}

\subsection{Homogenization to generalized stable-like dynamics}
\label{subsec:stable_like_generator}

We now describe the spatial structure of the effective generator. Throughout this subsection, set $\alpha:=2\beta\in(0,2)$ and retain $\kappa=\Gamma(1-\beta)/\mu_\tau$ from Theorem~\ref{thm:clock_limit}.

The following assumption ensures that the activity diffusion has a transition density with Gaussian two-sided bounds on finite time intervals, yielding the local stable-like bounds for the effective jump kernel.

\begin{assumption}[Uniform ellipticity]
\label{ass:uniform_ellipticity}
There exist constants $0<\lambda_*\le\Lambda_*<\infty$ such that
\begin{equation}
 \lambda_*|\xi|^2\le \xi^\top a_1(x)\xi\le\Lambda_*|\xi|^2,
 \qquad x,\xi\in\R^d.
 \label{eq:uniform_elliptic_activity}
\end{equation}
\end{assumption}

Under Assumptions~\ref{ass:slow_coefficients} and~\ref{ass:uniform_ellipticity}, the transition probabilities of the activity diffusion $V^x$ defined above admit a density $p_r^{(1)}(x,y)$ with respect to Lebesgue measure, so that
\[
 P_r^{(1)}f(x) =\int_{\R^d}f(y)p_r^{(1)}(x,y)\,\diff y, \qquad f\in C_0(\R^d).
\]
The notation $\mathcal A_{\alpha}$, $j_{\alpha}$, and $\ell_{\alpha}$ below suppresses the fixed dependence on $\Lop_1$ and $\kappa$. We define the nonlocal activity operator
\begin{equation}\label{eq:generalized_stable_definition}
 \mathcal A_{\alpha}f:=-\kappa(-\Lop_1)^{\alpha/2}f
 =\int_0^\infty(P_r^{(1)}f-f)\,\nu_\beta(\diff r),
 \qquad f\in\D.
\end{equation} 
The limit operator $\mathcal A_{\alpha}$ is an $\alpha$-stable-like activity operator, and has the following representation:
\begin{equation}\label{eq:generalized_stable_spatial}
\begin{aligned}
 \mathcal A_{\alpha}f(x) =\ell_{\alpha}(x)\cdot\nabla f(x) +\int_{\R^d\setminus\{x\}}\bigl[f(y)-f(x)- (y-x)\ind_{\{|y-x|\le1\}}\cdot\nabla f(x)\bigr]j_{\alpha}(x,y)\,\diff y,
\end{aligned}
\end{equation}
The drift is defined by the iterated integral
\begin{equation}\label{eq:generalized_stable_drift}
 \ell_{\alpha}(x) =\int_0^\infty \left[\int_{\R^d\setminus\{x\}} (y-x)\ind_{\{|y-x|\le1\}}p_r^{(1)}(x,y)\,\diff y\right] \nu_\beta(\diff r)\textcolor{red}{,}
\end{equation}
and its jump density is given by
\begin{equation}\label{eq:generalized_stable_kernel}
 j_{\alpha}(x,y)
 =\int_0^\infty p_r^{(1)}(x,y)\,\nu_\beta(\diff r),
 \qquad x\ne y.
\end{equation}
The inner integral over $y$ is evaluated first; cancellation in this integral ensures convergence, as verified in Section~\ref{subsec:fractional_topology}.

\begin{theorem}[Stable-like homogenization limit]
\label{cor:fractional_homog}
Under Assumptions~\ref{ass:slow_coefficients}, \ref{ass:local_activity}, \ref{ass:decoupling}, \ref{ass:level_C}, and~\ref{ass:uniform_ellipticity}, the microscopic diffusion homogenizes to a conservative Feller process with generalized $\alpha$-stable-like activity, where $\alpha=2\beta$. For every deterministic initial state, the convergence in Theorem~\ref{thm:full_homogenization}(i)--(ii) and the extended-time convergence \eqref{eq:main_extended_path_limit} hold. The effective generator is
\begin{equation}
 \overline{\Lop}f=\Lop_0f+\mathcal A_{\alpha}f
 =\Lop_0f-\kappa(-\Lop_1)^{\alpha/2}f,\qquad f\in\D.
 \label{eq:fractional_generator}
\end{equation}
The activity operator admits \eqref{eq:generalized_stable_spatial}, and its jump density satisfies, for constants $c,C>0$ independent of $x,y$,
\begin{equation}
 c|x-y|^{-d-\alpha}\le j_{\alpha}(x,y)
 \le C|x-y|^{-d-\alpha},\qquad 0<|x-y|\le1,
 \label{eq:local_stable_kernel_bounds}
\end{equation}
and $\sup_x\int_{\R^d\setminus\{x\}}(1\wedge|y-x|^2)j_{\alpha}(x,y)\,\diff y<\infty$.
\end{theorem}

The microscopic dynamics is local and continuous, while the nonlocal activity term emerges from the concentration of diffusive activity into vanishing time intervals.

\begin{remark}[Stable subordination and the Brownian case]
\label{rem:stable_brownian_specialization}
For $\Lop_1=\tfrac12\Delta$, the effective generator is $\Lop_0-\kappa2^{-\beta}(-\Delta)^\beta$. When $\Lop_0=0$, the limit is the symmetric $\alpha$-stable process $x+B_{S_t^\beta}$, with
\begin{equation}
 \overline{\Lop}=-\kappa\,2^{-\alpha/2}(-\Delta)^{\alpha/2},
 \qquad
 \E_0[e^{i\xi\cdot\bar X_t}]
 =\exp(-t\kappa2^{-\alpha/2}|\xi|^\alpha).
 \label{eq:fractional_laplacian}
\end{equation}
More generally, when $\Lop_0=0$, the limit is the activity diffusion at an independent stable clock; it need not be a symmetric stable L\'evy process.
\end{remark}

The following comparison situates the effective generator within the broader dichotomy between deterministic and heavy-tailed activity clocks.

\begin{remark}[Finite-mean versus heavy-tailed activity]
\label{thm:phase_transition}
Table~\ref{tab:phase_transition} compares the effective operators associated with different clock limits. The finite-mean row uses $a_T=T$, whereas the stable row uses the normalization in Assumption~\ref{ass:level_C}. The path descriptions refer to diffusion realizations with nontrivial activity; the jump limits need not have stationary independent increments.

\begin{table}[htbp]
\centering
\small
\renewcommand{\arraystretch}{1.3}
\resizebox{\linewidth}{!}{%
\begin{tabular}{lcccc}
\toprule
\textbf{Activity Regime} & \textbf{Activity Condition} & \textbf{Operational Clock} $A^\varepsilon$ & \textbf{Effective Generator} $\Lop_{\mathrm{eff}}$ & \textbf{Path Class} \\
\midrule
\textbf{Deterministic Averaging} & $Q \in L^1(\Pp)$, $\bar q = \E[Q] < \infty$ & $\bar q\,t$ (Deterministic) & $\Lop_0 + \bar q\,\Lop_1$ & Continuous Diffusion \\
\textbf{Heavy-Tailed Homogenization} & $R_1 \in RV_{-\beta}$, $0 < \beta < 1$ & $S_t^\beta$ (L\'evy Subordinator) & $\Lop_0 - \kappa(-\Lop_1)^\beta$ & Markov process with jumps \\
\bottomrule
\end{tabular}%
}
\caption{Effective operators for deterministic and random operational clocks.}
\label{tab:phase_transition}
\end{table}
\end{remark}

\subsection{Application to stochastic volatility}
\label{subsec:ou_volatility}

To illustrate the homogenization theory in financial mathematics, we examine a multiscale stochastic volatility model driven by the continuous Ornstein--Uhlenbeck environment developed in Section~\ref{sec:verification}.
In multiscale financial modeling (see e.g.~Pavliotis and Stuart~\cite{PavliotisStuart2008}), asset price volatility typically exhibits multiple timescales: a slow baseline scale reflecting macroeconomic conditions, and a fast scale driven by rapid order flow, news arrivals, or trading bursts.
Under this multiscale framework, consider an asset price $\mathsf S_t^\varepsilon$ governed by the Black--Scholes equation
\begin{equation}
 \frac{\diff\mathsf S_t^\varepsilon}{\mathsf S_t^\varepsilon}
 =\mu\,\diff t+\sqrt{v_0}\,\diff W_t^0
                  +\sqrt{v_1r_t^\varepsilon}\,\diff W_t^1,
 \qquad \mathsf S_0^\varepsilon=s_0>0,
 \label{eq:ou_volatility_price}
\end{equation}
where $\mu\in\R$ is the expected rate of return, $v_0\ge0$ represents a baseline variance, $v_1>0$ scales the fast volatility fluctuations, and $W^0,W^1$ are standard Brownian motions.
The instantaneous variance of the asset is $V_t^\varepsilon=v_0+v_1r_t^\varepsilon$, and its integrated variance is $v_0t+v_1A_t^\varepsilon$.

Applying It\^o's formula to the log-price $X_t^\varepsilon=\log\mathsf S_t^\varepsilon$ with initial state $x=\log s_0\in\R$, we obtain the microscopic log-price SDE
\begin{equation}
 \diff X_t^\varepsilon
 =\left(\mu-\frac{v_0}{2}-\frac{v_1}{2}r_t^\varepsilon\right)\diff t
   +\sqrt{v_0}\,\diff W_t^0
   +\sqrt{v_1r_t^\varepsilon}\,\diff W_t^1,
 \qquad X_0^\varepsilon=x.
 \label{eq:ou_volatility_micro}
\end{equation}
In classical multiscale models, the fast volatility factor is typically assumed to satisfy an $L^1$ ergodic theorem, so that $A_t^\varepsilon\to\bar q t$ deterministically as $\varepsilon\to0$, and the price homogenizes to a continuous diffusion with constant averaged volatility. In contrast, empirical financial returns display heavy-tailed volatility clustering and jump discontinuities.
Here, we model the fast rate by $r_\varepsilon$ in \eqref{eq:ou_activity_scaling}, driven by the stationary OU environment of Section~\ref{sec:verification}. As shown in \eqref{eq:ou_clock_limit}, its excursion rewards have heavy tails with index $\beta=1/2$, leading to the stable subordinator clock $A^\varepsilon\xRightarrow{M_1}S$ rather than deterministic averaging.
The main theorem applies to the log-price process \eqref{eq:ou_volatility_micro}, whose four slow coefficients are constant and satisfy Assumption~\ref{ass:slow_coefficients}; the Brownian motions $W^0,W^1$ are mutually independent and independent of the entire OU environment $Y$,  and satisfy Assumption~\ref{ass:decoupling}; the activity variance $v_1$ also satisfies Assumption~\ref{ass:uniform_ellipticity}.

Here the background and activity generators are
\begin{equation}
 \Lop_0=\left(\mu-\frac{v_0}{2}\right)\partial_x
                         +\frac{v_0}{2}\partial_{xx},\qquad
 \Lop_1=-\frac{v_1}{2}\partial_x+\frac{v_1}{2}\partial_{xx}.
 \label{eq:ou_volatility_generators}
\end{equation}
Theorems~\ref{thm:full_homogenization} and~\ref{cor:fractional_homog} give the homogenization limit with effective generator
\begin{equation}
 \overline{\Lop}f=\Lop_0f-\kappa_h(-\Lop_1)^{1/2}f,
 \qquad f\in C_c^\infty(\R).
 \label{eq:ou_volatility_effective}
\end{equation}
Because the two generators have constant coefficients, the limiting log-price has an explicit representation. Let $B^0,B^1,S$ be mutually independent, with $S$ as in \eqref{eq:ou_clock_limit}, and set $L_t^{\mathrm C}:=B_{S_t}^1$. This is a symmetric $\alpha$-stable L\'evy process with $\alpha=2\beta=1$ (a Cauchy process), since
\[
 \E[e^{i\xi L_t^{\mathrm C}}]
 =\E[e^{-\xi^2S_t/2}]
 =\exp\!\left(-t\frac{\kappa_h}{\sqrt2}|\xi|\right),
 \qquad \xi\in\R.
\]
The limiting log-price can therefore be written as
\begin{equation}
 \bar X_t=x+\left(\mu-\frac{v_0}{2}\right)t
       +\sqrt{v_0}\,B_t^0-\frac{v_1}{2}S_t
       +\sqrt{v_1}\,L_t^{\mathrm C},
 \label{eq:ou_volatility_limit}
\end{equation}
with macroscopic price limit $\bar{\mathsf S}_t=e^{\bar X_t}$ in finite-dimensional distributions.
Here $B^0$ is independent of $(L^{\mathrm C},S)$, while $L^{\mathrm C}$ and $S$ are coupled through the common subordinator $S$. Conditioning on $S$ yields the explicit characteristic function
\begin{equation}
 \E[e^{i\xi(\bar X_t-x)}]
 =\exp\left\{t\left[
 i\xi\left(\mu-\frac{v_0}{2}\right)-\frac{v_0\xi^2}{2}
 -\kappa_h\left(\frac{v_1}{2}(\xi^2+i\xi)\right)^{1/2}
 \right]\right\},\qquad \xi\in\R,
 \label{eq:ou_volatility_characteristic}
\end{equation}
where the square root is the principal branch.

This limit reveals two key features: First, although each microscopic path $t\mapsto \mathsf S_t^\varepsilon$ is strictly continuous, the macroscopic limit $\bar{\mathsf S}_t$ exhibits genuine jump discontinuities, showing that price jumps can emerge endogenously from continuous multiscale diffusion via heavy-tailed volatility bursts. Second, the operational drift $-v_1S_t/2$ originates from the microscopic It\^o correction $-\frac12 v_1 r_t^\varepsilon\diff t$; because it shares the heavy-tailed clock with the activity diffusion, it accumulates along operational time rather than physical time, coupling drift and dispersion into $(\xi^2+i\xi)^{1/2}$ and preserving geometric martingale compensation.

A key advantage of our semigroup homogenization framework is that it also applies when the background and activity generators do not commute, a setting where classical Bochner subordination and additive L\'evy decompositions break down.
To illustrate this non-commuting setting, keep the same environment, take $v_0>0$ and $0<\delta<1$, and replace the constant activity coefficients by
\begin{equation}
 \sigma_1(x)=\sqrt{v_1}(1+\delta\tanh x),\qquad
 b_1(x)=-\tfrac12\sigma_1(x)^2.
 \label{eq:ou_local_volatility_coefficients}
\end{equation}
The microscopic log-price now solves
\begin{equation}\label{eq:ou_local_volatility_micro}
 \diff X_t^\varepsilon=
 \left(\mu-\frac{v_0}{2}
       -\frac12r_t^\varepsilon\sigma_1(X_t^\varepsilon)^2\right)\diff t+\sqrt{v_0}\,\diff W_t^0
   +\sqrt{r_t^\varepsilon}\,\sigma_1(X_t^\varepsilon)\,\diff W_t^1.
\end{equation}
These slow coefficients are bounded and globally Lipschitz. Writing $a_1(x)=\sigma_1(x)^2$ gives
\[
 v_1(1-\delta)^2\le a_1(x)\le v_1(1+\delta)^2,\qquad
 \Lop_1=\frac{a_1(x)}2(\partial_{xx}-\partial_x).
\]
The coefficient of $\partial_{xxx}$ in $[\Lop_0,\Lop_1]$ is $v_0a_1'(x)/2$, which is not identically zero. Hence the background and activity generators do not commute. The same environmental verification and main theorems still yield
$\overline{\Lop}f=\Lop_0f-\kappa_h(-\Lop_1)^{1/2}f$ on $C_c^\infty(\R)$, with the generalized $1$-stable-like activity kernel of Theorem~\ref{cor:fractional_homog}. The limit is constructed by the extended-time evolution and semigroup composition in time order; the additive L\'evy representation \eqref{eq:ou_volatility_limit} applies only to the constant-coefficient case.

\begin{remark}[Noise decoupling and the leverage effect]
\label{rem:volatility_discussion}
Assumption~\ref{ass:decoupling} enforces independence between the diffusion noises $(W^0,W^1)$ and the environment $Y$, corresponding to a baseline market model without leverage effects. This decoupling isolates the purely temporal mechanism by which heavy-tailed intermittency generates price jumps. Incorporating correlation between return shocks and volatility innovations would require joint limit theorems for the clock and the driving stochastic integrals, representing an interesting direction for future research.
\end{remark}

\section{Proof of the homogenization theorem}
\label{sec:proofs}

Starting from the environmental clock limit in Theorem~\ref{thm:clock_limit}, we first construct the evolution in time order for finitely many bursts (Section~\ref{subsec:finite_bursts}). Extended-time representation and stability then provide the clock-to-state transfer (Section~\ref{subsec:extended_stability}). Section~\ref{subsec:main_proof} applies this transfer and restores the small bursts to identify the limiting semigroup and generator. Section~\ref{subsec:fractional_topology} proves the stable-like kernel estimates and checks the Brownian case; Throughout the proof, write $S=S^\beta$, $\nu=\nu_\beta$, and $\Phi(\lambda)=\kappa\lambda^\beta$.

\subsection{Finite-burst dynamics and ordered semigroup evolutions}\label{subsec:finite_bursts}

We first fix $\eta>0$, and retain the activity bursts whose sizes exceed
$\eta$. Since $\Lambda_\eta:=\nu((\eta,\infty))<\infty$,
these retained bursts form a compound-Poisson mechanism on every
finite physical-time interval. Between successive burst times, the
state evolves according to the background diffusion with semigroup
$P^{(0)}$. At a burst time $T_j$, physical time does not advance,
whereas the activity operational clock makes an instantaneous jump of
size $J_j>\eta$. During this operational-time increment, the state is
transported by the activity diffusion for time $J_j$, and hence the
burst acts on the state through the transition kernel
$P_{J_j}^{(1)}$. After the burst, the background evolution resumes
until the next burst.

It is useful to average over the burst size first. Define the Markov
operator
\begin{equation}
    K_\eta f
    :=
    \frac{1}{\Lambda_\eta}
    \int_{(\eta,\infty)} P_r^{(1)}f\,\nu(\diff r).
    \label{eq:K_eta}
\end{equation}
Thus, at rate $\Lambda_\eta$, one samples a burst size $J$ from the
probability $\Lambda_\eta^{-1}\nu(\diff r)\mathbf 1_{\{r>\eta\}}$ and then applies the activity transition $P_J^{(1)}$. In particular, the corresponding
jump part of the generator will be
\[
    \Lambda_\eta(K_\eta-I) = \int_{(\eta,\infty)}
        (P_r^{(1)}-I)\,\nu(\diff r).
\]
Let $(T_j,J_j)_{j=1}^m$ be the Poisson burst times and sizes in $[0,t]$ with $J_j>\eta$ and intensity $\diff s\,\nu(\diff r)$.
Conditional on the marked Poisson atoms, the evolution is obtained by an ordered composition of the background and activity semigroups:
\begin{equation}\label{eq:chronological_product}
    \cP_{\mathbf T,\mathbf J}(t)
    := P_{T_1}^{(0)}P_{J_1}^{(1)}P_{T_2-T_1}^{(0)}P_{J_2}^{(1)}\cdots
    P_{T_m-T_{m-1}}^{(0)}P_{J_m}^{(1)}P_{t-T_m}^{(0)},
\end{equation}
with $\cP_{\mathbf T,\mathbf J}(t)=P_t^{(0)}$ when $m=0$.
This representation is important because no commutation between semigroups $P^{(0)}$ and $P^{(1)}$ is assumed. Averaging $\cP_{\mathbf T,\mathbf J}(t)$ over the marked Poisson configuration produces the finite-burst semigroup studied below Proposition. The entire construction preserves the actual ordering of the
background and burst evolutions; at no point is commutativity of $P^{(0)}$ and $P^{(1)}$ used.

\begin{proposition}[Finite-burst semigroup and generator]
\label{thm:truncated_evolution}
Under Assumption~\ref{ass:slow_coefficients}, fix $\eta>0$, and let
$\Lambda_\eta$, $K_\eta$, and
$\cP_{\mathbf T,\mathbf J}(t)$ be defined as above. Then the averaging 
\begin{equation}\label{eq:finite_burst_semigroup}
    \bar P_t^\eta f := \E\!\left[\cP_{\mathbf T,\mathbf J}(t)f\right],
    \qquad f\in C_0(\R^d),
\end{equation}
defines a Feller semigroup on $C_0(\R^d)$ with conservative transition kernels. Its generator has domain $D(\Lop_\eta)=D(\Lop_0)$. Moreover, for $f\in D(\Lop_\eta)$, $\Lop_\eta$ is given by
\begin{equation}\label{eq:gen_eta}
    \Lop_\eta f = \Lop_0 f+\Lambda_\eta(K_\eta-I)f =\Lop_0 f
      +\int_{(\eta,\infty)}
        \bigl(P_r^{(1)}f-f\bigr)\,\nu(\diff r).
\end{equation}
\end{proposition}

\begin{proof}
Conditioning on the number, times, and marks of the retained Poisson
bursts gives
\begin{align*}
    \bar P_t^\eta f
    =
    e^{-\Lambda_\eta t}\sum_{m=0}^{\infty}&\int_{0<s_1<\cdots<s_m<t}\int_{(\eta,\infty)^m}
    P_{s_1}^{(0)}P_{r_1}^{(1)}
    P_{s_2-s_1}^{(0)}P_{r_2}^{(1)}\cdots P_{r_m}^{(1)} P_{t-s_m}^{(0)}f \prod_{j=1}^m \nu(\diff r_j)\,\diff s_1\cdots\diff s_m.
\end{align*}
Since both $P^{(0)}$ and $P^{(1)}$ are contractions, the
$m$th term is bounded by $e^{-\Lambda_\eta t}(\Lambda_\eta t)^m\|f\|_{\infty}/m!$.
Hence the series converges in $C_0(\R^d)$ and gives a positive
contraction.
Independent Poisson increments and concatenation of ordered semigroup evolutions give $\bar P_{t+s}^\eta=\bar P_t^\eta\bar P_s^\eta$. Moreover,
\[
 \|\bar P_h^\eta f-f\|_{\infty}
 \le\|P_h^{(0)}f-f\|_{\infty}+2\|f\|_{\infty}(1-e^{-\Lambda_\eta h})\longrightarrow0  \qquad \text{as } h \to 0.
\]
Thus this construction gives a Feller semigroup.
Moreover, each factor $P_s^{(0)}$ and $P_r^{(1)}$ has conservative transition kernels. Hence every ordered composition satisfies $\cP_{\mathbf T,\mathbf J}(t)(x,\R^d)=1$.
Averaging over the marked Poisson configuration preserves this
identity, so the transition kernel $\bar p_t^\eta$ of $\bar P_t^\eta$
satisfies $ \bar p_t^\eta(x,\R^d)=1$, which implies the conservative property of $\bar p_t^\eta$.

Let $G_\eta$ denote its generator and take
$f\in D(\Lop_0)$.  To compute $G_\eta f$, we decompose the short-time
evolution according to whether there are zero, one, or at least two
bursts in $[0,h]$. 
For zero bursts, we have
\[
     \frac{e^{-\Lambda_\eta h}P_h^{(0)}f-f}{h} = e^{-\Lambda_\eta h}
    \frac{P_h^{(0)}f-f}{h} + \frac{e^{-\Lambda_\eta h}-1}{h}\,f \longrightarrow
    \Lop_0 f-\Lambda_\eta f  \qquad \text{as } h \to 0.
\]
Similarly, the one-burst contribution in $[0,h]$ is, after setting the burst time equal to $hu$,
\[
    R_h^{(1)}f:=h\,e^{-\Lambda_\eta h}\int_0^1\int_{(\eta,\infty)}P_{hu}^{(0)}P_r^{(1)}P_{h(1-u)}^{(0)}f\, \nu(\diff r)\,\diff u .
\]
The contractivity of $P_t^{(0)}$ and $P_t^{(1)}$ gives the uniform
bound $\bigl\|P_{hu}^{(0)}P_r^{(1)}P_{h(1-u)}^{(0)}f\bigr\|_\infty \le \|f\|_\infty $.
Since $\nu((\eta,\infty))=\Lambda_\eta<\infty$, strong continuity of the semigroups and dominated convergence
therefore yields 
\begin{align*}
    \frac{R_h^{(1)}f}{h} \to
    \int_{(\eta,\infty)}P_r^{(1)}f\,\nu(\diff r)
    = \Lambda_\eta K_\eta f, \quad \text{as } h \to 0.
\end{align*}
Finally, let $N_h^\eta$ denote the number of retained bursts in $[0,h]$. Since the retained marked Poisson
random measure has intensity $\diff s\,\nu(\diff r)$ on $ [0,\infty)\times(\eta,\infty)$,
the number of its atoms in $[0,h]\times(\eta,\infty)$ is Poisson distributed with parameter $ \Lambda_\eta h$.
Hence
\begin{align*}
    \Pp(N_h^\eta\ge2)
    =
    1-\Pp(N_h^\eta=0)-\Pp(N_h^\eta=1) =
    1-e^{-\Lambda_\eta h}
      -\Lambda_\eta h\,e^{-\Lambda_\eta h}
    =
    O(h^2)
    \qquad\text{as }h\downarrow0.
\end{align*}
Since every ordered semigroup composition is contractive on
$C_0(\R^d)$, the contribution of trajectories containing at least two
bursts satisfies
\[
    \|R_h^{(\ge2)}f\|_\infty
    \le
    \|f\|_\infty\,\Pp(N_h^\eta\ge2)
    =
    O(h^2).
\]
Therefore, after division by $h$, the contribution of two or more bursts vanishes as $h\downarrow0$.
Consequently, we obtain $D(\Lop_0)\subset D(G_\eta)$ and the generator action in \eqref{eq:gen_eta} for $f\in D(\Lop_0)$. 

It remains only to identify the generator domain. Since $K_\eta$ is
a Markov operator, its contraction property implies $\|K_\eta\|_{\mathrm{op}}\le1$. The bounded perturbation appearing in the generator is
\[
 B_\eta:=\Lambda_\eta(K_\eta-I),\qquad
 \|B_\eta\|_{\mathrm{op}}\le2\Lambda_\eta.
\]
The preceding calculation therefore shows that $\Lop_0+B_\eta\subset G_\eta$, with $\Lop_0+B_\eta$ defined on $D(\Lop_0)$.
For $\lambda>2\Lambda_\eta$ and
$R_\lambda^{(0)}=(\lambda I-\Lop_0)^{-1}$,
contractivity of $P^{(0)}$ gives $\|R_\lambda^{(0)}\|_{\mathrm{op}}\le\lambda^{-1}$, and hence
\[
 \|B_\eta R_\lambda^{(0)}\|_{\mathrm{op}}
 \le\frac{2\Lambda_\eta}{\lambda}<1.
\]
The factorization
\[
 \lambda I-\Lop_0-B_\eta
 =(I-B_\eta R_\lambda^{(0)})(\lambda I-\Lop_0)
 \quad\hbox{on }D(\Lop_0)
\]
shows that $\lambda I-\Lop_0-B_\eta$ maps $D(\Lop_0)$ onto $C_0(\R^d)$: the first factor is invertible by the Neumann series, and the second has inverse $R_\lambda^{(0)}$. Explicitly, for every $g\in C_0(\R^d)$,
\[
 h=R_\lambda^{(0)}(I-B_\eta R_\lambda^{(0)})^{-1}g\in D(\Lop_0),
 \qquad (\lambda I-\Lop_0-B_\eta)h=g.
\]
To prove the reverse domain inclusion, take $f\in D(G_\eta)$ and put $g=(\lambda I-G_\eta)f$. Choose $h$ as above. Since $G_\eta$ extends $\Lop_0+B_\eta$,
\[
 (\lambda I-G_\eta)h
 =(\lambda I-\Lop_0-B_\eta)h
 =g=(\lambda I-G_\eta)f.
\]
As the generator of a contraction semigroup, $G_\eta$ has $(0,\infty)$ in its resolvent set. Thus $\lambda I-G_\eta$ is injective, which gives $f=h\in D(\Lop_0)$. Together with the first inclusion, this proves $D(G_\eta)=D(\Lop_0)$. Identifying $\Lop_\eta=G_\eta$ completes \eqref{eq:gen_eta}.
\end{proof}

\subsection{Extended-time representation and stability}\label{subsec:extended_stability}

Note that the microscopic equation is driven by two different notions of time: physical time $t$, which governs the background diffusion, and the accumulated activity clock $A_t^\varepsilon$, which governs the activity diffusion.
The extended-time construction places the background and activity-driven parts of the dynamics on a common continuous time scale. In the microscopic system, $A^\varepsilon$ is continuous and absolutely continuous, whereas its limiting counterpart is the c\`adl\`ag subordinator $S$ and may therefore have jumps. A direct stability argument for the original SDE is inconvenient, since the corresponding activity rates may become arbitrarily large as activity concentrates into bursts. We instead absorb these rates into a reparametrization of time. We first show that convergence of $A^\varepsilon$ at continuity points of $S$ yields locally uniform convergence of the extended clocks, and then prove stability of the extended-time equation.

Fix paths of $A^\varepsilon$ and $S$. Both extended clocks are nondecreasing, c\`adl\`ag, start at zero, and are finite on compact intervals. Define the total evolution times
\[
 F^\varepsilon(t):=t+A_t^\varepsilon,\qquad F(t):=t+S_t.
\]
The two components of this total time are recovered through
\begin{equation}
\begin{aligned}
    \tau^\varepsilon(u)
    &:=\inf\{t\ge0:F^\varepsilon(t)>u\},
    &c^\varepsilon(u)&:=u-\tau^\varepsilon(u),\\
    \tau(u)
    &:=
    \inf\{t\ge0:F(t)>u\},
    &c(u)&:=u-\tau(u).
\end{aligned}
    \label{eq:extended_clocks}
\end{equation}
Here $\tau(u)$ records the amount of physical time elapsed by
extended time $u$, and $c(u)$ records the amount of activity
operational time already consumed. The same interpretation applies to $(\tau^\varepsilon,c^\varepsilon)$.

\begin{lemma}\label{prop:extended_clocks}
The amount of physical time $\tau$, and the activity operational time $c$ are continuous, nondecreasing, and $1$-Lipschitz, with $\tau(u)+c(u)=u$. On $[t+S_{t-},t+S_t]$, physical time is frozen at $\tau(u)=t$ and activity increases by $\Delta S_t$. At the right endpoint, $\tau(F(t))=t$ and $c(F(t))=S_t$. The same statements hold for $(F^\varepsilon,\tau^\varepsilon,c^\varepsilon)$ with $A^\varepsilon$ in place of $S$.
\end{lemma}
\begin{proof}
Strict increase of $F$ gives a nondecreasing inverse. The inequality $F(t)-F(s)\ge t-s$ implies $0\le\tau(v)-\tau(u)\le v-u$. Hence $c(u)=u-\tau(u)$ is also nondecreasing and $1$-Lipschitz. The description of jump intervals and the endpoint identities follow directly from the definition of the generalized inverse.
\end{proof}

Thus the extended time $u$ is continuously allocated between physical evolution and activity evolution.
This parametrization also resolves the jumps of the limiting activity clock. If $S$ has a jump of size $r=\Delta S_t$ at physical time $t$, then $F$ jumps across the interval $[\,t+S_{t-},\,t+S_t\,]$.
On this interval, $\tau$ remains constant and equal to $t$, while $c$ increases by exactly $r$.  Hence physical time $t$ is frozen while the activity dynamics evolves for operational duration $r$.  In this way, an activity burst of vanishing physical duration is expanded into an ordinary interval of extended time.  This is the mechanism by which a limiting clock jump of size $r$ produces the activity transition $P_r^{(1)}$.

The next lemma supplies the clock convergence needed for SDE stability. Its proof is deterministic: convergence at continuity points must first be transferred to the inverse clocks, since $\tau(u)$ may itself be a jump time of $S$.

\begin{lemma}[Convergence of inverse clocks]
\label{lem:inverse_clock_convergence}
Let $\varepsilon_n\downarrow0$ and fix clock paths as above. If $A_t^{\varepsilon_n}\to S_t$ at every continuity point of $S$, then the corresponding inverse clocks satisfy $\tau^{\varepsilon_n}\to\tau$ and $c^{\varepsilon_n} \to c$ locally uniformly. This implication remains valid for any nondecreasing c\`adl\`ag approximating and limiting clocks that start at zero and are finite on compact intervals.
\end{lemma}
\begin{proof}
For fixed $u>0$ with $\tau(u)>0$, choose continuity points $t_-<\tau(u)<t_+$ of $S$ arbitrarily close to $\tau(u)$. Strict increase of $F$ gives $F(t_-)<u<F(t_+)$. Convergence at these two continuity points yields
\[
 F^{\varepsilon_n}(t_-)<u<F^{\varepsilon_n}(t_+)
 \quad\text{for all sufficiently large }n,
\]
so eventually $\tau^{\varepsilon_n}(u)\in[t_-,t_+]$. When $\tau(u)=0$, the upper bound and nonnegativity suffice; at $u=0$ all inverses vanish. The common $1$-Lipschitz bound upgrades pointwise convergence to uniform convergence on compact intervals by a finite grid argument. Indeed, for a grid $0=u_0<\cdots<u_m=U$ of mesh at most $\delta$,
\[
 \sup_{0\le u\le U}|\tau^{\varepsilon_n}(u)-\tau(u)|
 \le 2\delta+\max_{0\le k\le m}|\tau^{\varepsilon_n}(u_k)-\tau(u_k)|.
\]
First let $n\to\infty$ and then $\delta\downarrow0$. Finally, $c^{\varepsilon_n}(u)=u-\tau^{\varepsilon_n}(u)$. The proof uses only the stated path properties of the clocks.
\end{proof}

\paragraph{Extended-time equation and uniform estimates.}
Given independent Brownian motions $B^0$ and $B^1$, independent of the activity clocks, we therefore
associate with the clock $S$ the extended-time process
$Z^{x}=(Z_u^{x})_{u\ge0}$ defined by \eqref{eq:extended_sde}.
For $A^\varepsilon$, denote the solution by $Z^{\varepsilon,x}$ and replace $(\tau,c)$ by $(\tau^\varepsilon,c^\varepsilon)$ in \eqref{eq:extended_sde}. The following estimates apply to both equations and to the same equation driven by any extended clocks with the properties of Lemma~\ref{prop:extended_clocks}.
The background coefficients are driven by the physical clock
$\tau$, whereas the activity coefficients are driven by the
operational clock $c$.  Since $\diff\tau+\diff c=\diff u$ as measures on the extended-time axis, both clocks are dominated by ordinary time.  Consequently, the possibly large
microscopic activity rate no longer appears explicitly in the
coefficients of \eqref{eq:extended_sde}; it is encoded instead in how the unit amount of extended time is split between $\tau$ and $c$.
This yields moment and stability estimates whose constants are uniform over the admissible clocks.

\begin{lemma}[Uniform estimates for the extended-time equation]
\label{lem:extended_estimates}
Under Assumption~\ref{ass:slow_coefficients}, equation~\eqref{eq:extended_sde} has a unique nonexplosive strong solution. For $0\le u\le v\le U$, it satisfies
\begin{equation}
 \sup_x\E\sup_{u\le U}|Z_u^x-x|^2\le C(U+U^2),
 \qquad
 \sup_x\E|Z_v^x-Z_u^x|^4\le C_U|v-u|^2.
 \label{eq:extended_moments}
\end{equation}
Moreover, for solutions with the same clock and Brownian motions,
\[
 \sup_x\E|Z_v^x-Z_u^x|^2\le C_U|v-u|,
 \qquad
 \E\sup_{u\le U}|Z_u^x-Z_u^y|^2\le C_U|x-y|^2.
\]
The constants depend on the coefficient bounds, Lipschitz constants, and $U$, not on the clock. The same estimates hold conditionally on the clock and after averaging over random clocks independent of the Brownian motions.
\end{lemma}
\begin{proof}
Consider the usual augmentation of the filtration $\mathcal G_u =\mathcal C\vee\sigma\bigl(B^0_{\tau(v)},B^1_{c(v)}:0\le v\le u\bigr)$, where $\mathcal C:=\sigma\{\tau(v),c(v):v\ge0\}$ is the activity-clock $\sigma$-field.
Conditional on the clock, the two time-changed Brownian motions are $\mathcal G_u$ continuous independent martingales with brackets
\[
 \langle B^0\circ\tau\rangle_u=\tau(u)I_{d_0},\qquad
 \langle B^1\circ c\rangle_u=c(u)I_{d_1},\qquad
 \langle B^0\circ\tau,B^1\circ c\rangle_u=0.
\]
Here independence is conditional on $(\tau,c)$.
Let $M$ and $L$ be common bounds and Lipschitz constants for the four coefficients. All estimates below are conditional on $\mathcal C$.
Since both $\tau, c$ are $1$-Lipschitz and $\diff\tau+\diff c=\diff u$, the usual Picard iteration applies: start with $Z^{(0)}\equiv x$ and, for $k\ge0$, define
\begin{equation*}
 Z_u^{(k+1)} =x+\int_0^u b_0(Z_v^{(k)})\,\diff\tau(v)
       +\int_0^u b_1(Z_v^{(k)})\,\diff c(v) +\int_0^u\sigma_0(Z_v^{(k)})\,\diff B^0_{\tau(v)}
       +\int_0^u\sigma_1(Z_v^{(k)})\,\diff B^1_{c(v)}.
\end{equation*}
For $k\ge1$, set $D_k(u):=\E\!\left[\sup_{s\le u}|Z_s^{(k)}-Z_s^{(k-1)}|^2\,\middle|\,\mathcal C\right].$
Using Cauchy--Schwarz's inequality and the martingale maximal inequality, for $u\le U$ we have
\begin{align*}
 &\E\!\left[\sup_{s\le u}\left|
   \int_0^s[b_0(Z_v^{(k)})-b_0(Z_v^{(k-1)})]\,\diff\tau(v)
   \right|^2\,\middle|\,\mathcal C\right] \le UL^2\int_0^uD_k(v)\,\diff\tau(v),\\
 &\E\!\left[\sup_{s\le u}\left|
   \int_0^s[\sigma_0(Z_v^{(k)})-\sigma_0(Z_v^{(k-1)})]\,\diff B^0_{\tau(v)}
   \right|^2\,\middle|\,\mathcal C\right]\le4L^2\int_0^uD_k(v)\,\diff\tau(v).
\end{align*}
For the activity drift and martingale terms,
\begin{align*}
 &\E\!\left[\sup_{s\le u}\left|
   \int_0^s[b_1(Z_v^{(k)})-b_1(Z_v^{(k-1)})]\,\diff c(v)
   \right|^2\,\middle|\,\mathcal C\right] \le UL^2\int_0^uD_k(v)\,\diff c(v),\\
 &\E\!\left[\sup_{s\le u}\left|
   \int_0^s[\sigma_1(Z_v^{(k)})-\sigma_1(Z_v^{(k-1)})]\,\diff B^1_{c(v)}
   \right|^2\,\middle|\,\mathcal C\right] \le4L^2\int_0^uD_k(v)\,\diff c(v).
\end{align*}
Summing the four terms and using $\diff\tau+\diff c=\diff v$ yields
\[
 D_{k+1}(u) \le C_U\int_0^uD_k(v)\,[\diff\tau(v)+\diff c(v)] =C_U\int_0^uD_k(v)\,\diff v.
\]
Since $D_1(U)\le C M^2(U+U^2)$, after iteration we have a factorially summable bound on each finite interval
\[
 D_{k+1}(U)\le\frac{(C_UU)^k}{k!}D_1(U), \qquad\sum_{k\ge0}D_{k+1}(U)^{1/2}<\infty.
\]
Thus the iterates converge in $L^2(\Omega;C([0,U];\R^d))$ and uniformly almost surely. Using the Cauchy--Schwarz's inequality and the martingale isometry, we can pass the limit, and obtain the existence of the strong solutions.

To show the uniqueness and stability, for two solutions driven by the same clock and Brownian motions, set $H_{x,y}(u) :=\E\!\left[\sup_{s\le u}|Z_s^x-Z_s^y|^2\,\middle|\,\mathcal C\right]$. Then
\begin{align*}
 H_{x,y}(u) \le C|x-y|^2 +C_U\int_0^uH_{x,y}(v)\,[\diff\tau(v)+\diff c(v)] =C|x-y|^2+C_U\int_0^uH_{x,y}(v)\,\diff v.
\end{align*}
Gronwall's inequality gives $H_{x,y}(U)\le C_U|x-y|^2$; taking $x=y$ proves pathwise uniqueness.

Bounded coefficients and the martingale maximal inequalities give \eqref{eq:extended_moments} . For the displacement,
\begin{align*}
 \E\!\left[\sup_{s\le U}|Z_s^x-x|^2\,\middle|\,\mathcal C\right] \le4M^2[\tau(U)^2+c(U)^2] +16M^2[\tau(U)+c(U)] \le C M^2(U^2+U).
\end{align*}
For increments, write $\Delta\tau=\tau(v)-\tau(u)$ and $\Delta c=c(v)-c(u)$. The fourth-moment martingale inequality and $\Delta\tau+\Delta c=v-u$ give
\begin{align*}
 \E\!\left[|Z_v^x-Z_u^x|^4\,\middle|\,\mathcal C\right]
 &\le C M^4[(\Delta\tau)^4+(\Delta c)^4
                   +(\Delta\tau)^2+(\Delta c)^2]\\
 &\le C_U|v-u|^2,\\
 \E\!\left[|Z_v^x-Z_u^x|^2\,\middle|\,\mathcal C\right]
 &\le\left(\E\!\left[|Z_v^x-Z_u^x|^4\,\middle|\,\mathcal C\right]\right)^{1/2}
 \le C_U|v-u|.
\end{align*}
The uniformity in the clocks permits averaging over random clocks and prevents explosion on finite extended-time intervals.
\end{proof}

For later evaluation at random endpoints, the fourth-moment bound also gives a modulus estimate uniform in $x$ and in the clock. For $0<\gamma<1/4$, $0<\delta\le U$, and $\rho>0$,
\[
 \omega_{Z^x,U}(\delta)
 :=\sup_{\substack{u,v\in[0,U]\\|v-u|\le\delta}}|Z_v^x-Z_u^x|,
 \qquad
 \sup_x\Pp\bigl(\omega_{Z^x,U}(\delta)>\rho\bigr)
 \le C_{U,\gamma}\rho^{-4}\delta^{4\gamma}.
\]
Indeed, on the dyadic grid of $[0,U]$ with $2^j$ intervals, the largest adjacent increment has fourth moment at most $C_U2^{-j}$. Summing its $L^4$ norm with weights $2^{j\gamma}$ gives a finite bound on the $\gamma$-H\"older seminorm, since $\sum_j2^{-j(1/4-\gamma)}<\infty$. Chebyshev's inequality yields the displayed estimate, also after averaging over independent random clocks.

\begin{proposition}[Stability under convergence of extended clocks]
\label{prop:extended_stability}
Under Assumption~\ref{ass:slow_coefficients}, let $(\tau_n,c_n)$ and $(\tau,c)$ be deterministic extended clocks as in Lemma~\ref{prop:extended_clocks}, with $\tau_n\to\tau$ locally uniformly and $c_n(u)=u-\tau_n(u)$. In particular, Lemma~\ref{lem:inverse_clock_convergence} provides this convergence for $(\tau_n,c_n)=(\tau^{\varepsilon_n},c^{\varepsilon_n})$ whenever $A_t^{\varepsilon_n}\to S_t$ at every continuity point of $S$. Couple the solutions of \eqref{eq:extended_sde} using the same Brownian motions. Then, for every $U<\infty$,
\begin{equation}
 \sup_x\E\sup_{u\le U}|Z_u^{n,x}-Z_u^x|^2\longrightarrow0.
 \label{eq:extended_stability}
\end{equation}
The conclusion also holds for random clocks converging almost surely locally uniformly, when the entire family of clocks is independent of the Brownian motions.
\end{proposition}
\begin{proof}
We compare each solution with its Euler approximation, compare the two Euler recursions at fixed mesh, and finally remove the mesh. This avoids relying on uniform convergence of Brownian paths to pass directly through stochastic integrals.

For a mesh of size $0<h\le U$, we define the Euler approximation for the $n$th clock by
\[
 \begin{aligned}
 Z_u^{n,x,h}
 &=x+\int_0^u b_0(Z_{h\lfloor v/h\rfloor}^{n,x,h})\,\diff\tau_n(v)
       +\int_0^u b_1(Z_{h\lfloor v/h\rfloor}^{n,x,h})\,\diff c_n(v)\\
 &\quad+\int_0^u\sigma_0(Z_{h\lfloor v/h\rfloor}^{n,x,h})\,\diff B^0_{\tau_n(v)}
       +\int_0^u\sigma_1(Z_{h\lfloor v/h\rfloor}^{n,x,h})\,\diff B^1_{c_n(v)}.
 \end{aligned}
\]
We also define the limit-clock approximation $Z^{x,h}$ by the same formula without $n$. Thus $Z^{x,h}$ uses the coefficient value at the left endpoint $h\lfloor v/h\rfloor$ in each of its four integrals. The second-moment increment estimate of Lemma~\ref{lem:extended_estimates} gives
\[
 \sup_x\E|Z_v^x-Z_{h\lfloor v/h\rfloor}^x|^2\le C_Uh.
\]
For the $n$th equation the same estimate holds conditionally on $\mathcal C_n:=\sigma(\tau_n,c_n)$. To subtract the equations, set $e_u^{n,x,h}:=Z_u^{n,x}-Z_u^{n,x,h}$ and split the coefficient difference as
\[
  b_i(Z_v^{n,x})-b_i(Z_{h\lfloor v/h\rfloor}^{n,x,h})=[b_i(Z_v^{n,x})-b_i(Z_{h\lfloor v/h\rfloor}^{n,x})]  +[b_i(Z_{h\lfloor v/h\rfloor}^{n,x})-b_i(Z_{h\lfloor v/h\rfloor}^{n,x,h})],
\]
and similarly for $\sigma_i$, $i=0,1$. The first difference is the local increment error, while the second is bounded by $L|e_{h\lfloor v/h\rfloor}^{n,x,h}|$. 
Subtracting the Euler equation from \eqref{eq:extended_sde}, applying the martingale maximal inequality and Lipschitz continuity yields, as in the Picard estimate,
\begin{align*}
 &\E\!\left[\sup_{s\le u}|Z_s^{n,x}-Z_s^{n,x,h}|^2\,\middle|\,\mathcal C_n\right] \\
 \le & C_U\int_0^u
  \left\{\E\!\left[|Z_v^{n,x}-Z_{h\lfloor v/h\rfloor}^{n,x}|^2\,\middle|\,\mathcal C_n\right] +\E\!\left[\sup_{s\le v}|e_s^{n,x,h}|^2\,\middle|\,\mathcal C_n\right]\right\}
    [\diff\tau_n(v)+\diff c_n(v)] \\
 \le & C_U\int_0^u \E\!\left[\sup_{s\le  v}|e_s^{n,x,h}|^2\,\middle|\,\mathcal C_n\right]\,\diff v+C_Uh.
\end{align*}
Gronwall inequality therefore gives $\E\!\left[\sup_{s\le U}|e_s^{n,x,h}|^2\,\middle|\,\mathcal C_n\right]\le C_Uh$. Thus
\begin{equation}\label{eq:uniform_euler_error1}
 \sup_x\E\sup_{u\le U}|Z_u^{n,x}-Z_u^{n,x,h}|^2\le 2C_Uh.
\end{equation}
The same calculation applies to the limit clock, so that
\begin{equation}\label{eq:uniform_euler_error2}
 \sup_x\E\sup_{u\le U}|Z_u^x-Z_u^{x,h}|^2\le C_Uh.
\end{equation}
The same estimate holds for every $n$, with the same constant.

Let $N=\lceil U/h\rceil$ and $u_k=(kh)\wedge U$, $0\le k\le N$. With $\Delta_k\tau_n=\tau_n(u_{k+1})-\tau_n(u_k)$, and similarly for $c_n$, the grid recursion is
 \begin{align*}
  Z_{u_{k+1}}^{n,x,h} = & Z_{u_k}^{n,x,h} +b_0(Z_{u_k}^{n,x,h})\Delta_k\tau_n +b_1(Z_{u_k}^{n,x,h})\Delta_k c_n \\ 
    & +\sigma_0(Z_{u_k}^{n,x,h})[B^0_{\tau_n(u_{k+1})}-B^0_{\tau_n(u_k)}] +\sigma_1(Z_{u_k}^{n,x,h})
       [B^1_{c_n(u_{k+1})}-B^1_{c_n(u_k)}].
 \end{align*}
For fixed $h$, the Euler recursion is a finite composition of continuous functions of the increments of $\tau,c,B^0\circ\tau,B^1\circ c$. Uniform convergence of $\tau_n,c_n$ and continuity of the Brownian paths imply convergence of these increments, since all Brownian time arguments lie in the fixed interval $[0,U]$. Boundedness and Lipschitz continuity of the coefficients make this convergence uniform in the starting point $x$. Denote
\[
 V_n=(\tau_n,c_n,B^0_{\tau_n(\cdot)},B^1_{c_n(\cdot)}),\qquad
 V=(\tau,c,B^0_{\tau(\cdot)},B^1_{c(\cdot)}),
\]
For these grid points, let $\Delta_kV_n=V_n(u_{k+1})-V_n(u_k)$, and similarly for $V$. With common coefficient bound $M$ and Lipschitz constant $L$, the recursion gives
\[
 \sup_x|Z_{u_{k+1}}^{n,x,h}-Z_{u_{k+1}}^{x,h}| \le(1+2L|\Delta_kV_n|)\sup_x|Z_{u_k}^{n,x,h}-Z_{u_k}^{x,h}|+M\varepsilon_k^n,
\]
where $\varepsilon_k^n:=2|\Delta_kV_n-\Delta_kV|\longrightarrow0$ almost surely.
Iteration over finitely many grid points gives a bound independent of $x$ and tending to zero. The same argument controls the interpolation within each mesh interval, replacing $u_{k+1}$ by $u\in[u_k,u_{k+1}]$ and using uniform convergence of $V_n$ to $V$.
For domination, bounded coefficients give the simpler bound
\[
 \sup_x\sup_{u\le U}|Z_u^{n,x,h}-Z_u^{x,h}|
 \le G_h:=2M\Big(U+2N \Big(\sup_{s\le U}|B_s^0|+\sup_{s\le U}|B_s^1|\Big)\Big).
\]
Here $\E G_h^2<\infty$, uniformly in $n$ and in the clocks, for fixed $h$. Dominated convergence gives
\[
 \sup_x\E\sup_{u\le U}|Z_u^{n,x,h}-Z_u^{x,h}|^2\longrightarrow0.
\]
Then 
\begin{align*}
 &\sup_x\E\sup_{u\le U}|Z_u^{n,x}-Z_u^x|^2 \\
 \le & 3\sup_x\E\sup_{u\le U}|Z_u^{n,x}-Z_u^{n,x,h}|^2+3\sup_x\E\sup_{u\le U}|Z_u^{n,x,h}-Z_u^{x,h}|^2 + 3\sup_x\E\sup_{u\le U}|Z_u^{x,h}-Z_u^x|^2.
\end{align*}
Combining with \eqref{eq:uniform_euler_error1} and \eqref{eq:uniform_euler_error2}, letting $n\to\infty$ at fixed $h$ gives
\[
 \limsup_{n\to\infty}\sup_x\E\sup_{u\le U}|Z_u^{n,x}-Z_u^x|^2\le C_Uh.
\]
Then let $h\downarrow0$ to obtain \eqref{eq:extended_stability}. 
The proof remains valid for random clocks converging almost surely and independent of the Brownian motions, because the Euler bounds and the dominating Brownian moments are uniform in the clocks. Specifically, condition on the entire clock family to obtain the Euler error for each equation in its own clock filtration; the fixed-mesh comparison is pathwise, and its integrable bound permits averaging over the clocks.
\end{proof}

Now we recover $Z_{t+A_t^\varepsilon}^{\varepsilon,x}$ to the physical-time dynamics $X_t^{\varepsilon,x}$.

\begin{lemma}\label{lem:physical_extended_representation}
Under Assumptions~\ref{ass:slow_coefficients}, \ref{ass:local_activity}, and~\ref{ass:decoupling}, the microscopic SDE has a unique nonexplosive strong solution, and $X_t^{\varepsilon,x} \overset{\mathrm{law}}{=} Z_{t+A_t^\varepsilon}^{\varepsilon,x}$, where $Z^{\varepsilon,x}$ solves \eqref{eq:extended_sde} with clocks $(\tau^\varepsilon,c^\varepsilon)$. 
Conditionally on the microscopic clock $\mathcal C^\varepsilon:=\sigma(A_t^\varepsilon:t\ge0)$, for every $T<\infty$,
\begin{equation}
    \Law\!\left( (X_t^{\varepsilon,x})_{0\le t\le T}\,\middle|\,\mathcal C^\varepsilon \right)
    =
    \Law\!\left( (Z_{t+A_t^\varepsilon}^{\varepsilon,x})_{0\le t\le T} \,\middle|\,\mathcal C^\varepsilon \right)
\end{equation}
as probability measures on $C([0,T];\R^d)$.
\end{lemma}
\begin{proof}
Since $A^\varepsilon$ is absolutely continuous, change variables in the finite-variation and martingale integrals of \eqref{eq:extended_sde}. The physical martingale has quadratic variation $t$, the activity martingale has quadratic variation $A_t^\varepsilon$, and their cross variation is zero. Conditional on the environment, $B^1_{A_t^\varepsilon}$ has the law of the integral of $\sqrt{r_s^\varepsilon}$ against an independent Brownian motion, giving \eqref{eq:micro_sde}. The Picard argument of Lemma~\ref{lem:extended_estimates}, applied with control $\diff t+\diff A_t^\varepsilon$ and the prescribed microscopic Brownian motions, gives strong existence and pathwise uniqueness. Local finiteness of $t+A_t^\varepsilon$ prevents explosion. Conditional uniqueness then gives $X_t^{\varepsilon,x} \overset{\mathrm{law}}{=} Z_{t+A_t^\varepsilon}^{\varepsilon,x}$. Since $F^\varepsilon$ is continuous and strictly increasing onto $[0,\infty)$, $F^\varepsilon(\tau^\varepsilon(u))=u$. Composing the conditional path-law identity with $\tau^\varepsilon$ therefore identifies the canonical process $X^{\varepsilon,x}\circ\tau^\varepsilon$ with the solution $Z^{\varepsilon,x}$ of \eqref{eq:extended_sde}, jointly with the inverse clock.
\end{proof}

Define the limiting process by $\bar X_t^x=Z_{t+S_t}^x$.
It has c\`adl\`ag paths because $Z$ is continuous and $t+S_t$ is c\`adl\`ag and increasing. It is finite at finite times because $S_t<\infty$ almost surely.

\subsection{Passage to the limit and generator identification}
\label{subsec:main_proof}

We couple the activity measures, transfer their convergence through Lemma~\ref{lem:inverse_clock_convergence} and Proposition~\ref{prop:extended_stability}, and then identify the Feller limit by finite-burst approximation. The extended-time representation and the deterministic clock approximation needed for this transfer were established in Section~\ref{subsec:extended_stability}.

\begin{proof}[Proof of Theorem~\ref{thm:full_homogenization}]
Lemma~\ref{lem:physical_extended_representation} gives existence and uniqueness of the microscopic solution. We prove the convergence and identify the limiting dynamics in the following steps.

Apply the preceding construction to $A^\varepsilon$ and $S$. The random-measure space $\cM([0,\infty))$ is Polish under its vague topology. Along any sequence $\varepsilon_n\downarrow0$, a Skorokhod representation couples the activity measures with $\diff S$ so that they converge almost surely. At every deterministic time $t$ at which $S$ has no jump, this gives
$A_t^{\varepsilon_n}\to S_t$ almost surely. In particular this holds simultaneously at any rational times, and at any prescribed finite collection of observation times. 
On the same event, monotonicity extends the convergence from rational times to every continuity point of $S$. Lemma~\ref{lem:inverse_clock_convergence} therefore gives locally uniform convergence of the inverse clocks, and Proposition~\ref{prop:extended_stability} applies after adding independent Brownian motions.

Now we show the convergence of the extended-time parametrized paths.
By Lemma~\ref{lem:physical_extended_representation}, the coupled extended solutions have the joint laws of the canonical parametrizations defined before Theorem~\ref{thm:full_homogenization}. For every $U<\infty$, the inverse-clock convergence and Proposition~\ref{prop:extended_stability} give
\begin{equation}\label{eq:coupled_extended_path_convergence}
 \E\|\tau^{\varepsilon_n}-\tau\|_{\infty,[0,U]}^2
 +\sup_x\E\|Z^{\varepsilon_n,x}-Z^x\|_{\infty,[0,U]}^2
 \longrightarrow0.
\end{equation}
The first term tends to zero by almost-sure inverse-clock convergence and dominated convergence, since $0\le\tau^{\varepsilon_n}(u),\tau(u)\le u$. Thus the pairs $(\tau^{\varepsilon_n},Z^{\varepsilon_n,x})$ converge in probability on the coupling uniform metric 
\[d_U = \sup_{0\le u\le U}|\tau(u)-\widetilde\tau(u)| + \sup_{0\le u\le U}|Z_u-\widetilde Z_u|.\] 
Note that the metric $d_{\mathrm{loc}}=\sum_{m=1}^\infty2^{-m}(1\wedge d_m)$ induces the locally uniform topology. 
Truncating this series and using the preceding convergence for each integer $U=m$ gives convergence in probability under $d_{\mathrm{loc}}$, hence weak convergence of the original joint laws. Every null sequence admits this argument, and the limiting joint law is determined by the law of $S$ and the diffusion coefficients. This proves \eqref{eq:main_extended_path_limit} as $\varepsilon\downarrow0$.

Now we show the finite-dimensional convergence \eqref{eq:main_fdd_limit} and strong operator convergence.
Fix $t_1,\ldots,t_m$. The random endpoints $t_j+A_{t_j}^{\varepsilon_n}$ are tight, and converge to $t_j+S_{t_j}$. Localize these endpoints to $[0,U]$. On that event 
\[E_{n,U,\delta} := \left\{\max_{j\le m} (t_j+A_{t_j}^{\varepsilon_n})\vee (t_j+S_{t_j}) \le U\right\} \cap \left\{\max_{j\le m} \vert{}A_{t_j}^{\epsilon_n} - S_{t_j}\vert{} \le \delta \right\},\]
we have
\[
\max_{j\le m} |Z_{t_j+A_{t_j}^{\varepsilon_n}}^{\varepsilon_n,x}-Z_{t_j+S_{t_j}}^x| \le \|Z^{\varepsilon_n,x}-Z^x\|_{\infty,[0,U]} + \sup_{\substack{u,v\le U\\|u-v|\le\delta}}|Z_u^x-Z_v^x|,
\]
The first term in above inequality is bounded by the stability estimate \eqref{eq:extended_stability}. The fourth-moment estimate in \eqref{eq:extended_moments} makes the modulus of continuity $\sup_{\substack{u,v\le U\\|u-v|\le\delta}}|Z_u^x-Z_v^x|$ tight uniformly in $x$. Letting $n\to\infty$, then $U\to\infty$, gives
\begin{equation}
 \sup_x\E\left[1\wedge
 \max_{j\le m}|X_{t_j}^{\varepsilon_n,x}-\bar X_{t_j}^x|\right]\longrightarrow0
 \label{eq:coupled_fdd_convergence}
\end{equation}
on the coupling. This proves \eqref{eq:main_fdd_limit}.
Moreover, for every $f\in C_0(\R^d)$, $t \ge 0$, and $\bar P_tf(x)=\E[f(\bar X_t^x)]$, the convergence \eqref{eq:coupled_fdd_convergence} also gives
\begin{equation}
 \|P_t^\varepsilon f-\bar P_tf\|_{\infty}\longrightarrow0.
 \label{eq:strong_annealed_convergence}
\end{equation}
Every subsequence $\{(\tau^{\varepsilon_n},Z^{\varepsilon_n,x})\}_{n \ge 1}$ admits the same argument, and the limiting construction has a law determined only by the stable subordinator $S$, Brownian motions $B^0_t, B^1_t$, and the coefficients. Therefore these limits hold as $\varepsilon\downarrow0$, without subsequence extraction in the conclusion.

Now we show the limit operator $(\bar P_t)_{t\ge0}$ is a Markov semigroup with Feller property.
Let $S_t^\eta=\sum_{s\le t}\Delta S_s\ind_{\{\Delta S_s>\eta\}}$. Since the stable subordinator $S$ has finite jump variation on compact intervals without drift,
\[
 \sup_{t\le H}|S_t-S_t^\eta|
 =\sum_{s\le H}\Delta S_s\ind_{\{\Delta S_s\le\eta\}}
 \longrightarrow0\quad\hbox{almost surely}.
\]
Lemma~\ref{lem:inverse_clock_convergence} and Proposition~\ref{prop:extended_stability} thus prove
$\|\bar P_t^\eta f-\bar P_tf\|_{\infty}\to0$ for each $t \ge 0$ and $f\in C_0(\R^d)$.
For this finite-jump clock, denote its extended clocks by $(\tau^\eta,c^\eta)$, let $Z^{\eta,x}$ solve \eqref{eq:extended_sde} with these clocks, and set
$\bar X_t^{\eta,x}=Z_{t+S_t^\eta}^{\eta,x}$.
Fix $t>0$ and condition on the entire clock $S^\eta$. Almost surely, its jumps in $(0,t)$ can be listed as
$0<T_1<\cdots<T_m<t$, with sizes $J_j=\Delta S_{T_j}^\eta>\eta$; there is no jump at the fixed time $t$.
Thus $S_t^\eta=\sum_{i=1}^mJ_i$.

Between the $j$th and $(j+1)$st jumps, for $1\le j<m$, the corresponding extended-time interval is $\left[T_j+\sum_{i=1}^jJ_i,\ T_{j+1}+\sum_{i=1}^jJ_i\right]$.
On this interval,
\[
 \tau^\eta(u)=u-\sum_{i=1}^jJ_i,\qquad
 c^\eta(u)=\sum_{i=1}^jJ_i,\qquad \diff\tau^\eta(u)=\diff u,\qquad \diff c^\eta(u)=0.
\]
Consequently, with
$Y_v^{0,j}=Z_{T_j+\sum_{i=1}^jJ_i+v}^{\eta,x}$ and
$\widetilde B_v^{0,j}=B^0_{T_j+v}-B^0_{T_j}$, equation~\eqref{eq:extended_sde} reduces to
\[
 Y_v^{0,j}=Y_0^{0,j}
 +\int_0^v b_0(Y_w^{0,j})\,\diff w
 +\int_0^v\sigma_0(Y_w^{0,j})\,\diff\widetilde B_w^{0,j},
 \qquad 0\le v\le T_{j+1}-T_j.
\]
The conditional transition of this diffusion across interval $[0,T_{j+1}-T_j]$ is $P_{T_{j+1}-T_j}^{(0)}$.
When $m\ge1$, the initial background interval is $[0,T_1]$, where
$\tau^\eta(u)=u$ and $c^\eta(u)=0$; its transition is $P_{T_1}^{(0)}$.
The final background interval is $\left[T_m+\sum_{i=1}^mJ_i,\ t+\sum_{i=1}^mJ_i\right]$,
where $\tau^\eta(u)=u-\sum_{i=1}^mJ_i$ and $c^\eta(u)=\sum_{i=1}^mJ_i$; the same reduction gives the transition $P_{t-T_m}^{(0)}$.
If $m=0$, then $\tau^\eta(u)=u$ and $c^\eta(u)=0$ throughout $[0,t]$, and the whole evolution is $P_t^{(0)}$.

At the $j$th jump, for $1\le j\le m$, the extended-time interval is $\left[T_j+\sum_{i=1}^{j-1}J_i,\ T_j+\sum_{i=1}^jJ_i\right]$.
Here
\[
 \tau^\eta(u)=T_j,\qquad c^\eta(u)=u-T_j,
 \qquad \diff\tau^\eta(u)=0,\qquad \diff c^\eta(u)=\diff u.
\]
Hence, with
$Y_v^{1,j}=Z_{T_j+\sum_{i=1}^{j-1}J_i+v}^{\eta,x}$ and
$\widetilde B_v^{1,j}=B^1_{\sum_{i=1}^{j-1}J_i+v}-B^1_{\sum_{i=1}^{j-1}J_i}$, the same equation reduces to
\[
 Y_v^{1,j}=Y_0^{1,j}
 +\int_0^v b_1(Y_w^{1,j})\,\diff w
 +\int_0^v\sigma_1(Y_w^{1,j})\,\diff\widetilde B_w^{1,j},
 \qquad 0\le v\le J_j.
\]
Thus the activity diffusion runs for operational duration $J_j$, from
$Y_0^{1,j}=\bar X_{T_j-}^{\eta,x}$ to
$Y_{J_j}^{1,j}=\bar X_{T_j}^{\eta,x}$, and its conditional transition is $P_{J_j}^{(1)}$.

Conditionally on $S^\eta$, when $m\ge1$, the background Brownian time intervals
$[0,T_1]$, $[T_j,T_{j+1}]$ for $1\le j<m$, and $[T_m,t]$, as well as the activity Brownian time intervals $\left[\sum_{i=1}^{j-1}J_i,\ \sum_{i=1}^jJ_i\right]$, $1\le j\le m$,
are fixed; if $m=0$, the only interval is $[0,t]$ for the background motion.
Their increments of Brownian motions are independent of the preceding evolution, since the clock is independent of $B^0,B^1$, and the two Brownian motions are independent. Then iterating conditional expectations along the successive background and activity intervals, for $f\in C_0(\R^d)$ and $m\ge2$, we have
\begin{align*}
 \E\!\left[f(\bar X_t^{\eta,x})\,\middle|\,S^\eta\right]
 =\bigl(\cP_{\mathbf T,\mathbf J}(t)f\bigr)(x)=\Bigl(P_{T_1}^{(0)}P_{J_1}^{(1)}
 P_{T_2-T_1}^{(0)}P_{J_2}^{(1)}\cdots P_{T_m-T_{m-1}}^{(0)}P_{J_m}^{(1)}P_{t-T_m}^{(0)}f\Bigr)(x),
\end{align*}
which is defined by \eqref{eq:chronological_product}.
Taking expectation over the clock, whose retained jumps have Poisson intensity $\diff s\,\nu(\diff r)\ind_{\{r>\eta\}}$, yields
\[
 \E[f(\bar X_t^{\eta,x})]
 =\E\!\left[\bigl(\cP_{\mathbf T,\mathbf J}(t)f\bigr)(x)\right]
 =\bar P_t^\eta f(x),
\]
where the last equality is the definition \eqref{eq:finite_burst_semigroup}.

Consequently $\bar P_t$ maps $C_0$ into $C_0$, is positive and contractive, and inherits the semigroup identity. In detail,
\[
 \|\bar P_t^\eta\bar P_s^\eta f-\bar P_t\bar P_sf\|_{\infty}
 \le\|\bar P_s^\eta f-\bar P_sf\|_{\infty}
   +\|(\bar P_t^\eta-\bar P_t)\bar P_sf\|_{\infty}\longrightarrow0.
\]
The left-hand product equals $\bar P_{t+s}^\eta f$, proving $\bar P_{t+s}=\bar P_t\bar P_s$. Conservativity of probability follows from the finite-valued construction the limiting process $\bar X_t^x=Z_{t+S_t}^x$, rather than from testing a $C_0$-limit on the constant function $f(\cdot) \equiv 1$.

To prove strong continuity at $t=0$, fix $U>0$ and use
\[
 \sup_x\Pp(|\bar X_t^x-x|>z)
 \le\Pp(t+S_t>U)+z^{-2}C(U+U^2).
\]
First let $t\downarrow0$, using stochastic continuity of $S$ at zero, and then let $U\downarrow0$. Bounded uniform continuity of $f$ gives $\|\bar P_tf-f\|_{\infty}\to0$. Thus $(\bar P_t)_{t\ge0}$ is a Feller semigroup. 

For every finite times $0\le t_1<\cdots<t_k\le t$, we consider the test function $H(z):=h\bigl(z(t_1),\ldots,z(t_k)\bigr)$, 
where $h\in C_b((\R^d)^k)$. Then for every $f\in C_0(\R^d)$ and $t>0$, we have the identity
\[
 \E[H(\bar X^\eta)f(\bar X_{t+s}^\eta)]
 =\E[H(\bar X^\eta)\bar P_s^\eta f(\bar X_t^\eta)].
\]
Thus transition probabilities agrees with that of the limiting process $\bar X_t^x=Z_{t+S_t}^x $.
Since $f,\bar P_sf\in C_0(\R^d)$, finite-dimensional convergence as $\eta \to 0$ gives
\[ \E[H(\bar X^\eta)\bar P_sf(\bar X_t^\eta)]
 \longrightarrow\E[H(\bar X)\bar P_sf(\bar X_t)] .\]
Moreover, as $\eta \to 0$,
\begin{align*}
 \left|\E\!\left[H(\bar X^\eta)
 (\bar P_s^\eta f-\bar P_sf)(\bar X_t^\eta)\right]\right|\le\|H\|_\infty\|\bar P_s^\eta f-\bar P_sf\|_\infty
 \longrightarrow0.
\end{align*}
Passing to the limit in above identity yields $\E[H(\bar X)f(\bar X_{t+s})]
 =\E[H(\bar X)\bar P_sf(\bar X_t)]$.
With $\mathcal F_t^{\bar X}=\sigma(\bar X_r:0\le r\le t)$, by monotone-class theorem, 
$
 \E\!\left[f(\bar X_{t+s})\,\middle|\,\mathcal F_t^{\bar X}\right]
 =\bar P_sf(\bar X_t)$ almost surely,
which proves the Markov property of the constructed limit process $\bar X_t$.

Now we identify the generator of the limit process $\bar X_t$.
We first control the truncated generators, then pass to the limit in the semigroup identity. For $f\in C_c^\infty(\R^d)$, It\^o's formula for the autonomous activity diffusion and strong continuity give $f\in D(\Lop_1)$ and $\Lop_1f\in C_0$. The same holds for $\Lop_0$. Lemma~\ref{ass:semigroup_lipschitz} and contractivity give the integral converging in $\|\cdot\|_\infty$
\begin{equation}
 Lf=\Lop_0f+\int_0^\infty(P_r^{(1)}f-f)\,\nu(\diff r)\in C_0,
 \label{eq:proved_effective_operator}
\end{equation}
and the small-burst estimate
\begin{equation}
 \|\Lop_\eta f-Lf\|_{\infty}
 \le\|\Lop_1f\|_{\infty}\int_{(0,\eta]}r\,\nu(\diff r)\longrightarrow0.
 \label{eq:cauchy_bound}
\end{equation}
For the finite-jump semigroup $\bar P_t^\eta $, the generator identity gives
$\bar P_t^\eta f-f=\int_0^t\bar P_s^\eta\Lop_\eta f\,\diff s$.
The integrands converge in $\|\cdot\|_\infty$: their difference from $\bar P_sLf$ is at most
$\|\Lop_\eta f-Lf\|_{\infty}+\|(\bar P_s^\eta-\bar P_s)Lf\|_{\infty}$. They are bounded uniformly in $s$ and small $\eta$, so dominated convergence theorem yields $\bar P_tf-f=\int_0^t\bar P_sLf\,\diff s$.
Dividing by $t$ and using strong continuity proves $f\in D(\overline{\Lop})$ and $\overline{\Lop}f=Lf$. 
Using the Phillips formula \eqref{eq:phillips_prelim}, with zero subordinator drift, we identify the activity integral with $-\Phi(-\Lop_1)f$. For the regenerative environment, Theorem~\ref{thm:clock_limit} gives $\Phi(\lambda)=\kappa\lambda^\beta$ and $\nu=\nu_\beta$, yielding \eqref{eq:full_effective_gen}.
\end{proof}

\subsection{Stable-like kernels and the Brownian case}
\label{subsec:fractional_topology}

In this subsection, we prove the stable-like kernel representation in Theorem~\ref{cor:fractional_homog}. 

\begin{proof}[Proof of Theorem~\ref{cor:fractional_homog}]
The Phillips identification has been proved as part of Theorem~\ref{thm:full_homogenization}. The convergence assertions remains to identify the activity operator and its jump kernel in $\overline{\mathcal L}f$. For $0<\beta<1$, integration by parts, with $\lambda>0$, yields
\[
 \int_0^\infty(1-e^{-\lambda r})r^{-1-\beta}\,\diff r
 =\frac{\lambda}{\beta}\int_0^\infty e^{-\lambda r}r^{-\beta}\,\diff r
 =\frac{\Gamma(1-\beta)}{\beta}\lambda^\beta.
\]
The identity also holds at $\lambda=0$. With $\alpha=2\beta$ and $\kappa=\Gamma(1-\beta)/\mu_\tau$, the clock measure is
\[
 \nu_\beta(\diff r)=\frac{\kappa\beta}{\Gamma(1-\beta)}r^{-1-\beta}\,\diff r
 =\frac{\beta}{\mu_\tau}r^{-1-\beta}\,\diff r.
\]
Using the contractivity of semigroup $P^{(1)}$, for $f\in\D$, the identity $P_r^{(1)}f-f=\int_0^r P_s^{(1)}\Lop_1f\,\diff s$ gives
\begin{align*}
  \int_0^\infty\|P_r^{(1)}f-f\|_\infty\,\nu_\beta(\diff r)
  \le {}& \frac{\beta}{\mu_\tau}\left[
       \int_0^1 r\|\Lop_1f\|_\infty\frac{\diff r}{r^{1+\beta}}
       +\int_1^\infty\|P_r^{(1)}f-f\|_\infty\frac{\diff r}{r^{1+\beta}}\right] \\
  \le {}& \frac{\beta}{\mu_\tau}\left[
       \frac{\|\Lop_1f\|_\infty}{1-\beta}
       +\frac{2\|f\|_\infty}{\beta}\right]<\infty.
\end{align*}
Strong continuity of $P^{(1)}_t$ then makes the defining integral a convergent Bochner integral in $C_0(\R^d)$.
This proves the normalization in \eqref{eq:generalized_stable_definition} and the generator identity in Theorem~\ref{cor:fractional_homog}. 

We next derive the local kernel bounds from the diffusion heat kernel. The uniformly elliptic diffusion has a transition density with Gaussian bounds on every fixed finite time horizon.
In particular,$0<r\le 2$, by the Gaussian two-sided estimates (see e.g. Kusuoka~\cite{Kusuoka2017})
\begin{equation}\label{eq:elliptic_gaussian_bounds}
 c_1r^{-d/2}e^{-c_2|x-y|^2/r} \le p_r^{(1)}(x,y) \le C_1r^{-d/2}e^{-|x-y|^2/(C_2r)}.
\end{equation}
The positive constants in \eqref{eq:elliptic_gaussian_bounds} are independent on $x,y,r$.
For $r\ge1$, the semigroup identity and $\int p_{r-1}^{(1)}(x,z)\,\diff z=1$ for $r>1$ give
$p_r^{(1)}(x,y)=\int p_{r-1}^{(1)}(x,z)p_1^{(1)}(z,y)\,\diff z\le C_1$. Thus the part of \eqref{eq:generalized_stable_kernel} over $r\ge1$ is uniformly bounded by $C_1\nu_\beta([1,\infty))=C_1/\mu_\tau$.

Now we turn to show the jump kernel with $|x-y|\le1$  satisfies the bound \eqref{eq:local_stable_kernel_bounds}. 
Recall that
\[
j_{\alpha}(x,y)
=\int_{(0,1]}p_r^{(1)}(x,y)\,\nu_\beta(\diff r)
 +\int_{(1,\infty)}p_r^{(1)}(x,y)\,\nu_\beta(\diff r).
\]
We first consider small time contribution with $r\le 1$. If $0 < \rho:=|x-y|\le1$, substituting $u=\rho^2/r$ in \eqref{eq:elliptic_gaussian_bounds} and writing $m=(d+\alpha)/2>0$, we obtain the upper bound
\begin{align*}
   \int_0^1 p_r^{(1)}(x,y)\frac{dr}{r^{1+\beta}} \le & C_1\int_0^1 r^{-1-(d+\alpha)/2}e^{-\rho^2/(C_2r)}\,\diff r \\
 = & C_1\rho^{-d-\alpha}\int_{\rho^2}^\infty u^{m-1}e^{-u/C_2}\,\diff u \\
 \le & C_1C_2^m\Gamma(m)\rho^{-d-\alpha}.
\end{align*}
Since $\rho^{-d-\alpha}\ge1$, the large-time contribution for $r >1$ is absorbed into the same bound. This argument also proves that $j_{\alpha}(x,y)<\infty$ whenever $x\ne y$.
For the lower bound, integrate over $r\in[\rho^2,2\rho^2] \subset (0,2]$ in \eqref{eq:elliptic_gaussian_bounds}; the exponential is bounded below yields that
\[
 j_{\alpha}(x,y)
 \ge\frac{\beta}{\mu_\tau}c_1e^{-c_2}\int_{\rho^2}^{2\rho^2}r^{-1-m}\,\diff r
 =\frac{\beta}{\mu_\tau}c_1e^{-c_2}\frac{1-2^{-m}}{m}\rho^{-d-\alpha}.
\]
This proves \eqref{eq:local_stable_kernel_bounds}. 

To verify the integrability bound of jump kernel $j_{\alpha}(x,y)$, let $V^x$ be the autonomous activity diffusion with generator $\Lop_1$, started at $x$. Bounded coefficients in $\Lop_1$ and It\^o isometry imply
\[
 \E|V_r^x-x|^2
 \le2r\int_0^r\E|b_1(V_s^x)|^2\,\diff s
     +2\int_0^r\E|\sigma_1(V_s^x)|^2\,\diff s
 \le C(r^2+r).
\]
It follows that $\sup_x\E[1\wedge|V_r^x-x|^2]\le C(1\wedge r)$. Then by Fubini's theorem, we have
\begin{align}\label{eq:Levy bound}
 \sup_x \int_{\R^d\setminus\{x\}}(1\wedge|y-x|^2)j_{\alpha}(x,y)\,\diff y = & \sup_x \int_0^\infty\E[1\wedge|V_r^x-x|^2]
       \,\nu_\beta(\diff r) \nonumber \\
       \le & C\int_0^1 r^{-\alpha/2}\,\diff r
       +C\int_1^\infty r^{-1-\alpha/2}\,\diff r \nonumber\\
       < & \infty.
\end{align}
To verify the representation of $\mathcal A_{\alpha}$, we first choose a smooth cutoff $\chi:\R\to[0,1]$ with $\chi(s)=1$ for $s\le0$ and $\chi(s)=0$ for $s\ge1$, and set $h_n(z):=z\chi\bigl(n(|z|^2-1)\bigr), n\ge1$, so that $h_n\in C_c^\infty(\R^d;\R^d)$, $h_n(z)=z$ for $|z|\le1$, and $h_n(z)=0$ for $|z|\ge\sqrt{1+1/n}$. In particular, $|h_n(z)|\le\sqrt2$ and $h_n(z)\to h(z) = z\ind_{\{|z|\le1\}}(z)$ for every $z$.
For a drift estimate uniform in $n$, first apply It\^o's formula to the fixed smooth function $(h_1)_x(y)=h_1(y-x)$:
\[
 |\E h_1(V_r^x-x)| \le r\sup_x\|\Lop_1(h_1)_x\|_\infty\le Cr, \qquad 0<r\le1,
\]
where constant is independent of $x$. The differences $h_n-h_1$ and $h-h_1$ vanish on $|z|\le1$ and are uniformly bounded, so
\[
 |h_n(z)-h_1(z)|+|h(z)-h_1(z)|
 \le C(1\wedge|z|^2).
\]
The preceding displacement estimate therefore gives
\[
 |\E h_n(V_r^x-x)|+|\E h(V_r^x-x)|
 \le 2|\E h_1(V_r^x-x)|
      +C\E[1\wedge|V_r^x-x|^2]
 \le Cr,\qquad 0<r\le1,
\]
uniformly in $x,n$. For $r>1$, boundedness of $h_n$ and $h$ gives a uniform constant bound. Thus the drift in \eqref{eq:generalized_stable_drift} is well defined and satisfies
\begin{equation}\label{eq:absV}
\sup_x\int_0^\infty|\E h(V_r^x-x)|\,\nu_\beta(\diff r)
 \le\frac{\beta C}{\mu_\tau(1-\beta)}+\frac{\|h\|_\infty}{\mu_\tau}<\infty.
\end{equation}
For each fixed $r,x$, bounded convergence gives $\E h_n(V_r^x-x)\to\E h(V_r^x-x)$. The uniform bound $C(1\wedge r)$ is integrable against $\nu_\beta$, so dominated convergence also yields
\begin{align*}
 \ell_{h_n}(x):=\int_0^\infty\E h_n(V_r^x-x)\,\nu_\beta(\diff r)
       \longrightarrow\int_0^\infty\E h(V_r^x-x)\,\nu_\beta(\diff r)=\ell_{\alpha}(x).
\end{align*}
For $f\in C_c^\infty(\R^d)$, define $R_{f,n}(x,y)=f(y)-f(x)-h_n(y-x)\cdot\nabla f(x)$ and $R_f(x,y)=f(y)-f(x)-h(y-x)\cdot\nabla f(x)$. Since all these truncations equal $z$ on $|z|\le1$ and are uniformly bounded, Taylor's formula gives
\[
 |R_{f,n}(x,y)|+|R_f(x,y)|
 \le C_f(1\wedge|y-x|^2),
\]
with $C_f$ independent of $n$. The bound \eqref{eq:Levy bound} implies
\begin{equation}\label{eq:absRf}
\int_0^\infty\!\int_{\R^d}|R_f(x,y)|p_r^{(1)}(x,y)\,\diff y
       \,\nu_\beta(\diff r)
 \le C_f\int_{\R^d\setminus\{x\}}(1\wedge|y-x|^2)j_{\alpha}(x,y)\,\diff y<\infty.
\end{equation}
The same estimate holds with $R_f$ replaced by $R_{f,n}$, uniformly in $n$. For each smooth truncation $h_n$, the transition density gives
\begin{align*}
 P_r^{(1)}f(x)-f(x)
 &=\int_{\R^d}[f(y)-f(x)]p_r^{(1)}(x,y)\,\diff y\\
 &=\int_{\R^d}R_{f,n}(x,y)p_r^{(1)}(x,y)\,\diff y
   +\nabla f(x)\cdot\E h_n(V_r^x-x).
\end{align*}
Integrate this identity against $\nu_\beta(\diff r)$; absolute integrability and Fubini's theorem give the compensated term
\begin{align*}
 \int_0^\infty\!\int_{\R^d}R_{f,n}(x,y)p_r^{(1)}(x,y)\,\diff y\,
 \nu_\beta(\diff r) = & \int_{\R^d\setminus\{x\}}R_{f,n}(x,y)
 \left[\int_0^\infty p_r^{(1)}(x,y)
 \,\nu_\beta(\diff r)\right]\diff y\\
 &=\int_{\R^d\setminus\{x\}}R_{f,n}(x,y)j_{\alpha}(x,y)\,\diff y.
\end{align*}
Then combining with the uniform bound of drifts $\ell_{h_n}(x)$, we have
\[
 \mathcal A_{\alpha}f(x)
 =\ell_{h_n}(x)\cdot\nabla f(x)
  +\int_{\R^d\setminus\{x\}}R_{f,n}(x,y)j_{\alpha}(x,y)\,\diff y.
\]
Finally, $R_{f,n}(x,y)\to R_f(x,y)$, and the common quadratic bound is integrable against $j_{\alpha}(x,y)\,\diff y$. Then dominated convergence theorem together with the convergence of $\ell_{h_n}(x)$ above, yields
\begin{align*}
 \mathcal A_{\alpha}f(x) =\ell_{\alpha}(x)\cdot\nabla f(x)+\int_{\R^d\setminus\{x\}}\bigl[f(y)-f(x)-h(y-x)\cdot\nabla f(x)\bigr]
 j_{\alpha}(x,y)\,\diff y,
\end{align*}
which is \eqref{eq:generalized_stable_spatial}.
\end{proof}

\appendix

\section{Vague topology on Radon measures and cumulative clocks}\label{app:measure_topology}

To describe the concentration of activity into bursts, we use vague convergence of the measures $\mu^\varepsilon$ and then recover convergence of their cumulative clocks, which are used in Section \ref{sec:environmental_criteria}.

Let $E$ be a locally compact, second-countable Hausdorff space, and let $\cM(E)$ denote its nonnegative Radon measures. Writing $m(\varphi)=\int_E\varphi\,\diff m$, vague convergence means $m_n(\varphi)\to m(\varphi)$ for every continuous compactly supported function $\varphi\in C_c(E)$. The space $\cM(E)$ is Polish under this topology; see \cite{Kallenberg2017,Resnick1987}. We write $\cM_p(E)$ for the Radon point measures $\sum_i\delta_{x_i}$, and use $E=[0,\infty)$ for activity measures and $E=[0,\infty)\times(0,\infty]$ for burst point processes.

For random measures, $\mu_n\Longrightarrow\mu$ denotes weak convergence of their probability laws on $\cM(E)$. It is equivalent to joint convergence of $(\mu_n(\varphi_1),\ldots,\mu_n(\varphi_k))$ for every finite collection $\varphi_j\in C_c^+(E)$ to the corresponding vector for $\mu$, and also to convergence of Laplace functionals $\E[e^{-\mu_n(\varphi)}]\to\E[e^{-\mu(\varphi)}]$ for every $\varphi\in C_c^+(E)$ \cite[Theorem~4.11]{Kallenberg2017}.

A family of measures whose masses are uniformly bounded on every compact set is vaguely relatively compact \cite{Resnick1987}. Likewise, tightness of $(\mu_n(K))_n$ for every compact $K\subset E$ implies tightness of $(\mu_n)_n$ in $\cM(E)$ \cite[Theorem~4.10]{Kallenberg2017}. On $[0,\infty)$, it suffices to check the masses $\mu_n([0,H])$ for each $H>0$.

\paragraph{The Skorokhod $M_1$ topology.}
For $H>0$, let $D([0,H],\R)$ be the space of c\`adl\`ag paths. The completed graph $\Gamma_x$ of $x\in D([0,H],\R)$ contains the vertical segment joining $x(t-)$ to $x(t)$ at each time $t$, with $x(0-)=x(0)$. It is ordered by increasing time and, at a jump, from the left limit to the right value. Let $\mathcal R_H(x)$ consist of its continuous, order-preserving, onto parametrizations $(r,u):[0,1]\to\Gamma_x$, where $r$ and $u$ are the time and space coordinates. The $M_1$ metric is
\[
 d_{M_1,H}(x,y)
 =\inf_{\substack{(r,u)\in\mathcal R_H(x)\\(\widetilde r,\widetilde u)\in\mathcal R_H(y)}}
 \bigl(\|r-\widetilde r\|_{\infty}\vee\|u-\widetilde u\|_{\infty}\bigr).
\]
On $D([0,\infty),\R)$ we use the local $M_1$ topology: convergence means convergence of the restrictions to $[0,H]$ for every $H$ at which the limiting path is continuous. For nondecreasing paths, $M_1$ convergence on $[0,H]$ is equivalent to pointwise convergence on a dense set of continuity points together with convergence at $0$ and $H$; see \cite[Corollary~12.5.1 and Section~12.9]{Whitt2002}. This permits continuous increasing ramps to converge to jumps.

The following lemma transfers measure convergence to $M_1$ convergence of cumulative clocks.

\begin{lemma}\label{lem:measure_clock_continuity}
Let $m_n, m \in \cM([0,\infty))$ be nonnegative Radon measures.
\begin{enumerate}[label=(\alph*)]
\item If $m_n \to m$ vaguely and $m(\{t\}) = 0$ for some $t > 0$, then $m_n([0,t]) \to m([0,t])$. If additionally $m(\{0\}) = 0$, then $m_n(\{0\}) \to 0$. Consequently, if $m(\{0\})=m(\{H\})=0$ for some $H>0$, then
\[
 d_{M_1,H}\bigl(m_n([0,\cdot]),m([0,\cdot])\bigr)\longrightarrow0.
\]
\item Let $\mu_n, \mu$ be random measures on $[0,\infty)$ such that $\mu_n \Longrightarrow \mu$ in $(\cM([0,\infty)), \text{vague})$. If the limiting measure satisfies $\Pp(\mu(\{t\}) = 0) = 1$ for each fixed deterministic $t \ge 0$, then their cumulative clocks $A_t^n = \mu_n([0,t])$ converge to $A_t = \mu([0,t])$ in the local $M_1$ topology, and also in $M_1$ on every deterministic interval $[0,H]$:
\begin{equation}\label{eq:clock_fdd_from_measure}
 A^n  \Longrightarrow A \quad \text{in }(D([0,\infty),\R),M_1)
 \qquad \text{as } n\to\infty.
\end{equation}
Their finite-dimensional distributions converge as well.
\end{enumerate}
\end{lemma}
\begin{proof}
For $t>0$, assertion~(a) follows from vague Portmanteau \cite{Resnick1987} since $\partial[0,t]=\{t\}$ in $[0,\infty)$. At zero, choose $\varphi_\delta\in C_c([0,\infty))$ with $0\le\varphi_\delta\le1$, $\varphi_\delta(0)=1$, and support in $[0,\delta)$. Then
\[
 \limsup_n m_n(\{0\})\le\int\varphi_\delta\,\diff m
 \le m([0,\delta))\longrightarrow m(\{0\})=0.
\]
For the final assertion of (a), the cumulative functions are nondecreasing and c\`adl\`ag. The preceding limits give convergence at $0$, at $H$, and at every continuity point in $(0,H)$ of $m([0,\cdot])$. The monotone-path criterion in \cite[Corollary~12.5.1]{Whitt2002} therefore gives the claimed $M_1$ convergence.

For (b), consider the cumulative map $F(m)=m([0,\cdot])$. This map is Borel measurable, since its coordinate evaluations are measurable and generate the Borel sigma-field of the Skorokhod space. Part~(a), applied at every continuity horizon $H$, shows that $F$ is continuous from the vague topology to the local $M_1$ topology at every $m$ with $m(\{0\})=0$. The continuous mapping theorem gives \eqref{eq:clock_fdd_from_measure}. For each fixed $H>0$, the additional almost-sure condition $\mu(\{H\})=0$ gives the same conclusion for the restriction to $[0,H]$. Evaluation at a positive time is $M_1$-continuous at paths continuous there, and evaluation at zero is continuous; the no-fixed-atom condition thus also gives finite-dimensional convergence.

Finally, stochastic continuity of a subordinator gives $\Delta S_t=0$ almost surely for each fixed $t>0$, while $\diff S(\{0\})=0$ by the convention $S_0=0$. Hence the required no-fixed-atom condition holds.
\end{proof}

\section*{Acknowledgments}
This paper arose in collaboration with AI tools. 
The main ideas, definitions, questions, and main direction of this work were developed by the author. 
AI tools were used to help with checking calculations and arguments, finding references, and improving the exposition.
All mathematical statements, proofs,and conclusions were verified by the author.

\end{document}